\documentclass[pdflatex,sn-mathphys-num]{sn-jnl}

\hypersetup{hidelinks}
\usepackage{comment}
\usepackage{caption}
\usepackage{subcaption}
\usepackage{enumitem}
\usepackage{graphicx}%
\usepackage{multirow}%
\usepackage{amsmath,amssymb,amsfonts}%
\usepackage{mathtools,amsthm}%
\mathtoolsset{showonlyrefs}
\usepackage{mathrsfs}%
\usepackage[title]{appendix}%
\usepackage{xcolor}%
\usepackage{textcomp}%
\usepackage{manyfoot}%
\usepackage{booktabs}%
\usepackage{algorithm}%
\usepackage{algorithmicx}%
\usepackage{algpseudocode}%
\usepackage{listings}%
\usepackage[siunitx]{circuitikz}
\usepackage{tikz}
\usetikzlibrary{calc,patterns,
                 decorations.pathmorphing,
                 decorations.markings}
\usetikzlibrary{shapes,arrows,patterns,positioning}
\tikzstyle{block} = [draw, fill=gray!20, rectangle, 
    minimum height=2em, minimum width=4em]
\tikzstyle{sum} = [draw, fill=gray!20, circle, node distance=1.5cm]
\tikzstyle{input} = [coordinate]
\tikzstyle{output} = [coordinate]
\tikzstyle{pinstyle} = [pin edge={to-,thin,black}]

\DeclareMathOperator{\rank}{rank}
\DeclareMathOperator{\tr}{tr}
\DeclareMathOperator{\rad}{rad}
\DeclareMathOperator{\cent}{cent}
\DeclareMathOperator{\diam}{diam}
\DeclareMathOperator{\aff}{aff}

\newcommand{\norm}[1]{\left|#1\right|}

\theoremstyle{thmstyleone}
\newtheorem{theorem}{Theorem}

\newtheorem{definition}[theorem]{Definition}

\newtheorem{lemma}[theorem]{Lemma}
\newtheorem{proposition}[theorem]{Proposition}

\begin{document}

\title[Chebyshev Centers and Radii for Sets Induced by Quadratic Matrix Inequalities]{Chebyshev Centers and Radii for Sets Induced by Quadratic Matrix Inequalities}


\author*[1]{\fnm{Amir} \sur{Shakouri}}\email{a.shakouri@rug.nl}

\author[1]{\fnm{Henk} \sur{J. van Waarde}}\email{h.j.van.waarde@rug.nl}

\author[1]{\fnm{M. Kanat} \sur{Camlibel}}\email{m.k.camlibel@rug.nl}

\affil[1]{\orgdiv{Bernoulli Institute for Mathematics, Computer Science and Artificial Intelligence}, \orgname{University of Groningen}, \orgaddress{\street{Nijenborgh 9}, \city{Groningen}, \postcode{9747AG}, \country{The Netherlands}}}


\abstract{This paper studies sets of matrices induced by quadratic inequalities. In particular, the center and radius of a smallest ball containing the set, called a \emph{Chebyshev center} and the \emph{Chebyshev radius}, are studied. In addition, this work studies the \emph{diameter} of the set, which is the farthest distance between any two elements of the set. Closed-form solutions are provided for a Chebyshev center, the Chebyshev radius, and the diameter of sets induced by quadratic matrix inequalities (QMIs) with respect to arbitrary unitarily invariant norms. Examples of these norms include the Frobenius norm, spectral norm, nuclear norm, Schatten $p$-norms, and Ky Fan $k$-norms. In addition, closed-form solutions are presented for the radius of the largest ball \emph{within} a QMI-induced set. Finally, the paper discusses applications of the presented results in data-driven modeling and control. }




\maketitle

\section{Introduction}

Quadratic matrix inequalities (QMIs) are ubiquitous in systems and control theory. Prominent examples are the Lyapunov and the algebraic Riccati inequality. The study of QMI-induced sets is classically motivated by the design of optimal and robust controllers \cite{willems1971least,stoorvogel1990quadratic,scherer1999lecture,scherer2001lpv,skelton2017unified,scherer2006matrix}, and recently, has seen a renewed interest due to its applications in data-driven modeling and control \cite{van2020noisy,miller2022data,rapisarda2023orthogonal,van2023informativity}.  
Latest studies on certain types of QMI-induced sets have led to the development of matrix versions of several well-known results such as Yakubovich's S-lemma \cite{van2020noisy,van2023quadratic} and Finsler's lemma \cite{van2021matrix,meijer2024unified}. Research on QMI-induced sets is also relevant for matrix regression and system identification where the noise model may be given by quadratic inequalities. 

Given a compact set of matrices and a norm, the center and radius of a smallest ball containing the set are referred to as a \emph{Chebyshev center} and the \emph{Chebyshev radius}, respectively. In addition, the farthest distance between two elements within the set is referred to as the \emph{diameter} of the set. Such notions have applications in the approximation of matrices within QMI-induced sets. For instance, in the context of data-driven modeling of dynamical systems, given that the system matrices are unknown but belong to a QMI-induced set obtained from noisy data \cite{van2020noisy}, a Chebyshev center is an estimation for the \emph{true} system matrices, and the Chebyshev radius is the worst-case error of such an estimation. Since a Chebyshev center is a point that has the least worst-case error, it is also called a \emph{best worst-case} estimation \cite{beck2007regularization}. In this case, one may consider a Chebyshev center as the nominal system and implement tools from robust control theory to design controllers that stabilize all systems in a ball
around the nominal system. The reader can refer to \cite{alimov2021geometric} and the references therein for more details on Chebyshev centers and their applications in other fields.

The complexity of finding a Chebyshev center and the Chebyshev radius depends on the characteristics of the set and the given norm. For instance, if the set is an ellipsoid within a Euclidean space, equipped with the Euclidean norm, then the Chebyshev center is unique and is simply the center of the ellipsoid with the Chebyshev radius being the length of the largest semi-axis. However, for an intersection of ellipsoids, there are in general no closed-form solutions and existing algorithms only converge to a proximity of the Chebyshev center \cite{beck2007regularization,eldar2008minimax}.  The effect of the given norm on the computational complexity of finding the Chebyshev radius is evident as a norm can itself be NP-hard to compute (such as matrix $p$-norms when $p\neq 1,2,\infty$ \cite{hendrickx2010matrix}). To find a Chebyshev center, there are a few nontrivial cases where a polynomial time algorithm exists for specific norms. For instance, finding a Chebyshev center for polyhedra can be reduced to a linear program \cite{alimov2021geometric}, or, when the set is finite, one may find the smallest enclosing ball in linear time \cite{welzl2005smallest}.  In the absence of polynomial time algorithms for finding a Chebyshev center, approximate solutions may be sought. For example, for the intersection of two ellipsoids, a relaxed Chebyshev center can be found by solving a semi-definite program \cite{beck2007regularization}, which is equal to the true Chebyshev center if the set is defined on the complex field. To the authors' knowledge, a full characterization of the Chebyshev centers and radii for QMI-induced sets with respect to arbitrary unitarily invariant matrix norms is still missing in the existing literature.

In this paper, we present closed-form solutions for a Chebyshev center, the Chebyshev radius, and the diameter of a certain class of compact QMI-induced sets. In addition, we study the radius of a largest ball within a QMI-induced set, and for the cases where the set has an empty interior, we present the radius of a largest lower-dimensional ball contained in the set. All the provided solutions are with respect to arbitrary unitarily invariant matrix norms, which are expressed in terms of symmetric gauge functions \cite{marshall2011inequalities}. Examples of such norms are the Schatten $p$-norms and the Ky Fan $k$-norms, which in particular, include spectral, nuclear, and Frobenius norms. The freedom over the choice of the norm makes it possible to approximate a QMI-induced set with respect to the norm that is most relevant to the target application. 

The presented results in this paper are applied to the modeling and control of unknown linear time-invariant systems based on noisy input-output data. The noise is assumed to be unknown but belongs to a known bounded set \cite{fogel1979system,belforte1990parameter,fogel1982value,giarre1997model,garulli1999tight,bai1999bounded}. In this setting, the input-output data and knowledge about the noise give rise to the \emph{set of data-consistent systems}, which includes all systems that could have generated the data. In the case that the set of data-consistent systems is bounded, a \emph{set-membership identification} problem seeks a solution to minimize the worst-case estimation error, that is, to find a Chebyshev center and the Chebyshev radius for the set of data-consistent systems \cite{milanese1991optimal,milanese1985optimal,eldar2008minimax,wu2013new}. In this setting, the Chebyshev radius represents the identification accuracy. This accuracy needs to be computed with respect to a given norm that is most suitable for the target application. For instance, the application of the nuclear norm in system identification has been explored in \cite{liu2010interior}. The application of the spectral norm in the design of robust controllers has been studied in \cite{khargonekar1990robust}. The relevance of the Frobenius norm in characterizing uncertainties in the system parameters has been shown in \cite{lee1996quadratic,boukas1999h}. Nevertheless, except for some well-known cases, such as minimizing the Frobenius norm of the data error via the least-squares problem, system identification seldom offers closed-form solutions. This motivates exploring cases for which a closed-form solution exists, which may lead to valuable information on how the input-output data affect the identification error with respect to different norms. This, in turn, can bridge the gap between direct and indirect data-driven control. 

The remainder of this paper is organized as follows. Section \ref{sec:II} is devoted to preliminaries. Section \ref{sec:IV} includes the main results and Section \ref{sec:III} discusses their applications in data-driven modeling and control. Finally, Section \ref{sec:VII} concludes the paper. 

\section{Preliminaries}
\label{sec:II}

\subsection{Notation}


For $a,b\in\mathbb{Z}$ with $a\leq b$, let $[a,b]\coloneqq\{x\in\mathbb{Z}:a\leq x\leq b\}$. Let $\mathbb{R}_+^n$ denote the set of $n$-dimensional real vectors with nonnegative entries. We define \linebreak $\mathbb{R}^n_{\downarrow}\coloneqq\{x\in\mathbb{R}_+^n:x_i\geq x_{i+1},i=1,\ldots,n\}$. For $x\in\mathbb{R}^n$ and $k\leq n$ we define the vector containing the first $k$ entries of $x$ as $x_{[k]}\coloneqq\begin{bmatrix}
x_1 & \cdots & x_k
\end{bmatrix}^\top$, and for $k>n$ we define $x_{[k]}\coloneqq\begin{bmatrix}
x^\top & 0
\end{bmatrix}^\top\in\mathbb{R}^k$. For two vectors $x,y\in\mathbb{R}^n$, we say $y$ \emph{weakly majorizes} $x$, denoted by $x\prec_w y$, if $\sum_{i=1}^k x_i\leq \sum_{i=1}^k y_i$ for all $k\in[1,n]$. We define $e_{1,n}\coloneqq\begin{bmatrix}
1 & 0 & \cdots & 0
\end{bmatrix}^\top\in\mathbb{R}^n$, which is simply denoted by $e_1$ when its size is clear from the context.

The vector of singular values of matrix $M\in\mathbb{C}^{n\times m}$, arranged in a nonincreasing order, is denoted by $\sigma(M)\in\mathbb{R}^{\min\{n,m\}}_{\downarrow}$. Let $\sigma_*(M)$ denote the smallest nonzero singular value of $M\in\mathbb{C}^{n\times m}$ if $M\neq 0$ and $\sigma_*(M)=0$ if $M=0$. We denote the kernel of matrix $M\in\mathbb{R}^{n\times m}$ by $\ker M\coloneqq \{x\in\mathbb{R}^m:Mx=0\}$. The Moore-Penrose pseudo-inverse of a real matrix $M$ is denoted by $M^\dagger$. Moreover, we denote the Hadamard (or Schur) product of two matrices $M,N\in\mathbb{R}^{n\times m}$ by $M\circ N$. 

Let $\mathbb{S}^{n}$ denote the set of $n\times n$ symmetric matrices. A matrix $M\in\mathbb{S}^{n}$ is called \emph{positive definite} (resp., \emph{positive semi-definite}) and denoted by $M>0$ (resp., $M\geq 0$) if all its eigenvalues are positive (resp., nonnegative). By $M<0$ (resp., $M\leq 0$) we imply that $-M>0$ (resp., $-M\geq 0$).

A norm $\norm{\ \cdot\ }$ on $\mathbb{R}^{p\times q}$ is said to be \emph{essentially strictly convex} if the unit ball in $\mathbb{R}^{p\times q}$ with respect the norm is a strictly convex set, or equivalently, $X_1,X_2\in\mathbb{R}^{p\times q}$ and $\norm{X_1+X_2}=\norm{X_1}+\norm{X_2}$ imply $X_1=\alpha X_2$ for some $\alpha\geq0$ (cf. \cite[p. 106]{holmes2006course}, \cite[eq. 2.8]{ziketak1988characterization}). 

For a set of matrices $\mathcal{X}\subseteq\mathbb{R}^{p\times q}$ and matrices $X_0,M,N$ of appropriate sizes, we define $X_0+\mathcal{X}\coloneqq\{X_0+X:X\in\mathcal{X}\}$, $\mathcal{X}-\mathcal{X}\coloneqq\{X_1-X_2:X_1,X_2\in\mathcal{X}\}$, and $M \mathcal{X}N\coloneqq \{M XN:X\in\mathcal{X}\}$. We denote the affine hull of the set $\mathcal{X}$ \linebreak by $\aff\mathcal{X}\coloneqq\left\{\sum_{i=1}^n\alpha_iX_i:X_i\in\mathcal{X},\alpha_i\in\mathbb{R},n\in\mathbb{N},\sum_{i=1}^n \alpha_i=1\right\}$. 

\subsection{Center, radius and diameter of sets}

Let $\mathcal{X}\subset\mathbb{R}^{p\times q}$ be a compact set and let $\norm{\ \cdot\ }$ be a norm on $\mathbb{R}^{p\times q}$. We define
\begin{equation}
\label{eq:diam}
\diam \mathcal{X}\coloneqq \max_{\Delta\in\mathcal{X}-\mathcal{X}} \norm{\Delta} 
\end{equation} 
as the \emph{diameter} of $\mathcal{X}$, 
\begin{equation}
\label{eq:cheb}
\rad \mathcal{X}\coloneqq \min_{C\in\mathbb{R}^{p\times q}} \max_{X\in\mathcal{X}} \norm{C-X}
\end{equation}
as the \emph{Chebyshev radius} of $\mathcal{X}$, and
\begin{equation}
\label{eq:cent}
\cent\mathcal{X}\coloneqq \left\{C\in\mathbb{R}^{p\times q}:\norm{C-X}\leq\rad\mathcal{X}\ \text{for all}\  X\in\mathcal{X}\right\}
\end{equation}
as the \emph{set of Chebyshev centers}\footnote{For sets that are not compact, one can define the Chebyshev centers and radii by taking the infimum and supremum instead of the maximum and minimum, see \cite{alimov2019chebyshev}.}\textsuperscript{,}\footnote{Note that Chebyshev centers and radii are defined differently in some references such as \mbox{\cite[Chapter 
8.5.2]{boyd2004convex}}, where the Chebyshev center is referred to as the center of a largest ball that lies within the set. } of $\mathcal{X}$. Due to compactness of $\mathcal{X}$, $\diam\mathcal{X}$ and $\rad\mathcal{X}$ are well-defined, and $\cent\mathcal{X}$ is nonempty. 

Based on the above definitions, we have the following results. 
\begin{proposition}
\label{prop:chebuniq}
Let $\mathcal{X}\subset\mathbb{R}^{p\times q}$ be a compact set. Then:
\begin{enumerate}[label=\normalfont{(\alph*)},ref=\ref{prop:chebuniq}(\alph*)]
    \item\label{prop:chebuniq-a} $\tfrac{1}{2}\diam\mathcal{X}\leq\rad\mathcal{X}\leq\diam\mathcal{X}$.
    \item\label{prop:chebuniq-b} $\cent \mathcal{X}$ is compact and convex with no interior points.
    \item\label{prop:chebuniq-bc} For every $X_0\in\mathbb{R}^{p\times q}$ we have  $\diam(X_0+\mathcal{X})=\diam\mathcal{X}$, $\rad(X_0+\mathcal{X})=\rad\mathcal{X}$, and $\cent(X_0+\mathcal{X})=X_0+\cent\mathcal{X}$.
    \item\label{prop:chebuniq-c} If the norm is essentially strictly convex, then the set $\mathcal{X}$ has a unique Chebyshev center.
\end{enumerate}
\end{proposition}
\begin{proof}
(a) Observe that from the triangle inequality $\tfrac{1}{2}\norm{X_1-X_2}\leq \tfrac{1}{2}\norm{C-X_1}+\tfrac{1}{2}\norm{C-X_2}$. Therefore, taking $C$ to be a Chebyshev center, we have $\tfrac{1}{2}\norm{X_1-X_2}\leq \rad\mathcal{X}$ for all \mbox{$X_1,X_2\in\mathcal{X}$}, which implies $\tfrac{1}{2}\diam\mathcal{X}\leq\rad\mathcal{X}$. The second inequality is obvious from the definitions \eqref{eq:diam} and \eqref{eq:cheb}. 

(b) To prove the boundedness of $\cent\mathcal{X}$, we note that for any $C\in\cent\mathcal{X}$ and $X\in\mathcal{X}$ we have $\norm{C}\leq \norm{C-X}+\norm{X}\leq \rad\mathcal{X}+\norm{X}$. Since $\mathcal{X}$ is bounded, we see that $\cent\mathcal{X}$ is bounded. To prove closedness, consider a sequence $C_i\in\cent\mathcal{X}$, $i\in\mathbb{N}$, with $\lim_{i\rightarrow\infty} C_i=C$. Let $X\in\mathcal{X}$. Note that $\norm{C-X}\leq\norm{C-C_i}+\norm{C_i-X}\leq \norm{C-C_i}+\rad\mathcal{X}$. Take the limit as $i\rightarrow \infty$ to conclude that $\norm{C-X}\leq \rad\mathcal{X}$. Thus, $C\in\cent\mathcal{X}$ and therefore $\cent\mathcal{X}$ is closed. 
To prove convexity, suppose that $C_1,C_2\in\cent\mathcal{X}$. Then, by the triangle inequality, we have the following for any $\alpha\in[0,1]$:
\begin{equation}
\label{eq:prop:chebuniq-1}
\norm{\alpha C_1+(1-\alpha)C_2-X}\leq\alpha\norm{C_1-X}+(1-\alpha)\norm{C_2-X}\leq\rad\mathcal{X}.
\end{equation}
Therefore, for any $C_1,C_2\in\cent\mathcal{X}$ and any $\alpha\in[0,1]$ we have $\alpha C_1+(1-\alpha)C_2\in\cent\mathcal{X}$, which implies $\cent\mathcal{X}$ is convex. Now, we prove that $\cent\mathcal{X}$ has no interior points. Assume, on the contrary, that the interior of $\cent\mathcal{X}$ is nonempty. Let $C_1$ belong to the interior of $\cent\mathcal{X}$. Let $X\in\mathcal{X}$ be such that $|C_1-X|=\rad\mathcal{X}$. Since the interior of $\cent\mathcal{X}$ is nonempty, there exists a sufficiently small $\alpha>0$ such that $C_2=C_1+\alpha(C_1-X)\in\cent\mathcal{X}$. This implies that $|C_2-X|=(\alpha+1)|C_1-X|=(\alpha+1)\rad\mathcal{X}$. Since $\rad\mathcal{X}>0$, we have $|C_2 - X| > \rad \mathcal{X}$, meaning that $C_2$ is not a Chebyshev center, thus resulting in a contradiction. Therefore, $\cent\mathcal{X}$ does not have an interior point.

(c) For the diameter, the proof immediately follows from the fact that \mbox{$(X_0+\mathcal{X})-(X_0+\mathcal{X})=\mathcal{X}-\mathcal{X}$}. For the radius, we have
\begin{equation}
\rad (X_0+\mathcal{X})= \min_{C\in\mathbb{R}^{p\times q}} \max_{X\in\mathcal{X}} \norm{C-X-X_0}=\min_{C\in\mathbb{R}^{p\times q}+X_0} \max_{X\in\mathcal{X}} \norm{C-X},
\end{equation}
which is equal to $\rad \mathcal{X}$ as $\mathbb{R}^{p\times q}+X_0=\mathbb{R}^{p\times q}$. For the set of centers, we have 
\begin{equation}
\begin{split}
\cent(X_0+\mathcal{X})&=\{C:\norm{C-X}\leq\rad(X_0+\mathcal{X}) \ \text{for all}\  X\in X_0+\mathcal{X}\} \\
&= \{C:\norm{C-X-X_0}\leq\rad\mathcal{X} \ \text{for all}\  X\in\mathcal{X}\}\\
&=\{C+X_0:\norm{C-X}\leq\rad\mathcal{X} \ \text{for all}\   X\in\mathcal{X}\}=X_0+\cent\mathcal{X}.
\end{split}
\end{equation}

(d) Suppose that $C_1,C_2\in\cent\mathcal{X}$ and the norm $\norm{\ \cdot\ }$ is essentially strictly convex. We show that $C_1=C_2$. For this, recall from part (b) that $\cent\mathcal{X}$ is convex. Thus, we have $\tfrac{1}{2}(C_1+C_2)\in\cent\mathcal{X}$. Hence, there exists an $X\in\mathcal{X}$ such that
\begin{equation}
\label{eq:pf_equal_rad}
\norm{\tfrac{1}{2}(C_1+C_2)-X}=\rad\mathcal{X}.
\end{equation}
Moreover, we observe that $\norm{C_i-X}\leq\rad\mathcal{X}$ for $i=1,2$. Thus, it follows from the triangle inequality that
\begin{equation}
\label{eq:pf_inequal_rad}
\norm{\tfrac{1}{2}(C_1+C_2)-X}\leq\tfrac{1}{2}\norm{C_1-X}+\tfrac{1}{2}\norm{C_2-X}\leq\rad\mathcal{X}.
\end{equation}
From \eqref{eq:pf_equal_rad} and \eqref{eq:pf_inequal_rad} we have
\begin{equation}
\label{eq:prop:chebuniq-2}
\norm{\tfrac{1}{2}(C_1-X)+\tfrac{1}{2}(C_2-X)}=\tfrac{1}{2}\norm{C_1-X}+\tfrac{1}{2}\norm{C_2-X}=\rad\mathcal{X}.
\end{equation}
Since we have $\norm{C_i-X}\leq\rad\mathcal{X}$ for $i=1,2$, it follows from \eqref{eq:prop:chebuniq-2} that 
\begin{equation}
\label{eq:prop:chebuniq-3}
\norm{C_1-X}=\norm{C_2-X}=\rad\mathcal{X}.
\end{equation}
As the norm is essentially strictly convex, \eqref{eq:prop:chebuniq-2} implies that $C_1-X=\alpha(C_2-X)$ for some $\alpha\geq 0$. It follows now from \eqref{eq:prop:chebuniq-3} that $\alpha=1$. This implies that $C_1=C_2$.
\end{proof}

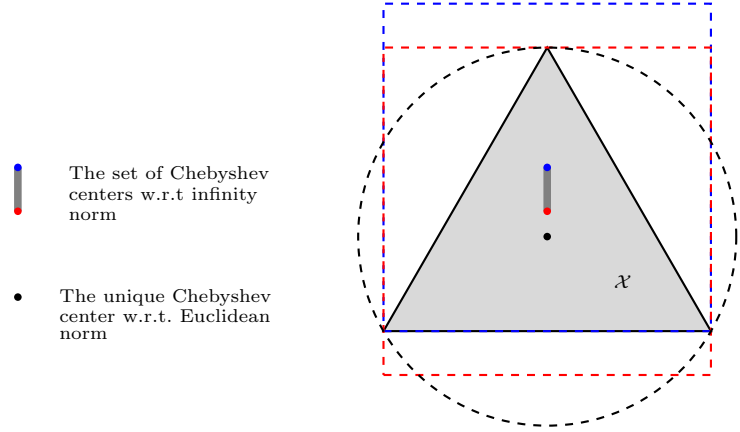
\begin{figure}[h]
    \centering
    \begin{tikzpicture}[thick, scale=1]
    \draw[fill=gray!30,,draw=black] (-2.165,-1.25) -- (90:2.5) -- (2.165,-1.25) -- cycle;
    \draw[blue,dashed] (-2.165,-1.25) -- (2.165,-1.25) -- (2.165,3.08) -- (-2.165,3.08) -- (-2.165,-1.25) ;
    \draw[red,dashed] (-2.165,-1.83) -- (2.165,-1.83) -- (2.165,2.5) -- (-2.165,2.5) -- (-2.165,-1.83) ;
    \draw[dashed] (0,0) circle (2.5cm);
    \draw[fill=black,draw=none] (0,0) circle (0.5mm);
    \draw[gray,line width=1mm] (0,0.335) -- (0,0.915) ;
    \draw[fill=blue,draw=none] (0,0.915) circle (0.5mm);
    \draw[fill=red,draw=none] (0,0.335) circle (0.5mm);
    \draw[fill=black,draw=none] (-7,-0.8) circle (0.5mm);
    \node[align=left] at (-5,-0.9) {\\ \footnotesize The unique Chebyshev \vspace{-0.2cm} \\ \footnotesize center w.r.t. Euclidean \vspace{-0.2cm} \\ \footnotesize norm};
    \draw[gray,line width=1mm] (-7,0.335) -- (-7,0.915) ;
    \draw[fill=blue,draw=none] (-7,0.915) circle (0.5mm);
    \draw[fill=red,draw=none] (-7,0.335) circle (0.5mm);
    \node[align=left] at (-5,0.7) {\\ \footnotesize The set of Chebyshev\vspace{-0.2cm} \\ \footnotesize  centers w.r.t infinity\vspace{-0.2cm} \\ \footnotesize  norm};
    \node[align=center] at (1,-0.6) {\footnotesize $\mathcal{X}$};
    \end{tikzpicture}
    \caption{The Chebyshev radius and the set of Chebyshev centers for the closed area defined by an equilateral triangle in terms of Euclidean and infinity norms.}
    \label{fig:1}
\end{figure}

As an example, consider the set $\mathcal{X}\subset\mathbb{R}^2$ to be the closed shaded area defined by an equilateral triangle as depicted in Fig. \ref{fig:1}. With respect to the Euclidean norm, which is an essentially strictly convex norm, the set of Chebyshev centers is a singleton shown as a black dot. However, with respect to the infinity norm, both points shown in red and blue are Chebyshev centers because they are both centers of a smallest square that contains the triangle. Hence, the set of Chebyshev centers is not a singleton, but the line segment shown in dark gray. This line segment is the center of all squares that contain the triangle and have the same radius as the blue and red ones. The diameter of this set is the length of the triangle's side for both Euclidean and infinity norms.

\subsection{Sets induced by quadratic matrix inequalities}

We define
\begin{equation}
\pmb{\Pi}_{q,p}\!\coloneqq\! \left\{\!\begin{bmatrix}
\Pi_{11} & \Pi_{12} \\
\Pi_{21} & \Pi_{22}
\end{bmatrix}\!\in\!\mathbb{S}^{q+p}\!:\! \Pi_{11}\!\in\!\mathbb{S}^q,\Pi_{22}\!\in\!\mathbb{S}^p,\Pi_{22}\leq 0,\Pi|\Pi_{22}\geq 0,
\ker\Pi_{22}\!\subseteq\!\ker\Pi_{12}\!\right\}\!,
\end{equation}
where $\Pi|\Pi_{22}\coloneqq \Pi_{11}-\Pi_{12}\Pi_{22}^\dagger \Pi_{21}$ is the generalized Schur complement of $\Pi$ with respect to $\Pi_{22}$. 

For a given $\Pi\in\pmb{\Pi}_{q,p}$, the matrices $\Pi_{11}\in\mathbb{S}^q$, $\Pi_{12}\in\mathbb{R}^{q\times p}$, $\Pi_{21}=\Pi_{12}^\top$, and $\Pi_{22}\in\mathbb{S}^p$ are defined by
\begin{equation}
\begin{bmatrix}
\Pi_{11} & \Pi_{12} \\
\Pi_{21} & \Pi_{22}
\end{bmatrix}=\Pi.
\end{equation}
We also define
\begin{equation}
\pmb{\Pi}_{q,p}^-\coloneqq \{\Pi\in\pmb{\Pi}_{q,p}:\Pi_{22}<0\}.
\end{equation}

We define the QMI-induced set $\mathcal{Z}_p(\Pi)$ associated with a matrix $\Pi\in\pmb{\Pi}_{q,p}$ to be
\begin{equation}
\label{eq:Zp}
\mathcal{Z}_p(\Pi)\coloneqq \left\{Z\in\mathbb{R}^{p\times q}:\begin{bmatrix}
I \\ Z
\end{bmatrix}^\top \Pi \begin{bmatrix}
I \\ Z
\end{bmatrix}\geq 0\right\}.
\end{equation}

The basic properties of $\mathcal{Z}_p(\Pi)$ are recalled
in the following proposition. 

\begin{proposition}[{\cite[Thms. 3.2 and 3.3]{van2023quadratic}}]
\label{prop:qmi}
Let $\Pi\in\pmb{\Pi}_{q,p}$. Then:
\begin{enumerate}[label=\normalfont{(\alph*)},ref=\ref{prop:qmi}(\alph*)]
    \item\label{prop:qmi-a} $\mathcal{Z}_p(\Pi)$ is nonempty and convex.
    \item\label{prop:qmi-b} $\mathcal{Z}_p(\Pi)$ is bounded if and only if $\Pi\in\pmb{\Pi}_{q,p}^-$.
    \item\label{prop:qmi-c} Suppose that $\Pi\in\pmb{\Pi}_{q,p}^-$. Then, $Z\in\mathcal{Z}_p(\Pi)$ if and only if
    \begin{equation}
    \label{eq:prop:char-1}
    Z=-\Pi_{22}^{-1} \Pi_{21} + (-\Pi_{22})^{-\frac{1}{2}} S (\Pi|\Pi_{22})^{\frac{1}{2}}
    \end{equation}
    for some $S\in\mathbb{R}^{p\times q}$ satisfying $SS^\top\leq I$. 
\end{enumerate}
\end{proposition}

As an example, consider the three-dimensional ellipsoid $\tfrac{z_1^2}{a^2}+\tfrac{z_2^2}{b^2}+\tfrac{z_3^2}{c^2}=1$, depicted in Fig. \ref{fig:2}. The region enclosed by this ellipsoid, $\tfrac{z_1^2}{a^2}+\tfrac{z_2^2}{b^2}+\tfrac{z_3^2}{c^2}\leq1$, is a QMI-induced set with $Z^\top=\begin{bmatrix}
z_1 & z_2 & z_3
\end{bmatrix}$ and $\Pi$ being a diagonal matrix with diagonal entries of $1$, $-\tfrac{1}{a^2}$, $-\tfrac{1}{b^2}$, and $-\tfrac{1}{c^2}$. In this example, \eqref{eq:prop:char-1} is equivalent to the parametrization of the ellipsoid in the spherical coordinates. 

\begin{figure}[h]
    \centering
    \begin{tikzpicture}
    \draw[rotate=-10,fill=gray!10] (0,0) ellipse (3cm and 1cm);
    \draw[dashed,rotate=-10] (3cm,0)  arc (0:180:3cm and 0.5cm);
    \draw[rotate=-10] (-3cm,0)  arc (180:360:3cm and 0.5cm);
    \draw[rotate=-10] (0,1cm) arc (90:270:0.5cm and 1cm);
    \draw[dashed,rotate=-10] (0,-1cm) arc (-90:90:0.5cm and 1cm);
    \draw[rotate=-10,densely dotted] (0,0)--(3cm,0);
    \draw[rotate=-10,->] (0,0)--(0.25cm,0);
    \draw[rotate=-10,densely dotted] (0,0)--(0,1cm);
    \draw[rotate=-10,->] (0,0)--(0,0.25cm);
    \draw[densely dotted] (0,0)--(-0.5cm,-0.4cm);
    \draw[->] (0,0)--(-0.1666cm,-0.1333cm);
    \draw[fill=black,draw=none] (0,0) circle (0.5mm);
    \node[align=center] at (0.2,-0.2) {\footnotesize $z_3$};
    \node[align=center] at (-0.3,0) {\footnotesize $z_2$};
    \node[align=center] at (0.25,0.2) {\footnotesize $z_1$};
    \end{tikzpicture}
    \caption{The closed shaded area defined by an ellipsoid is an example of a QMI-induced set.}
    \label{fig:2}
\end{figure}
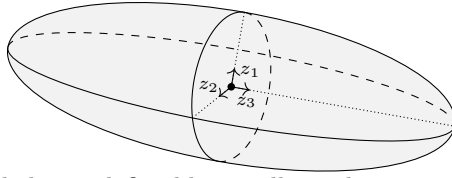

\subsection{Unitarily invariant norms}

In this paper, we focus on unitarily invariant matrix norms.

\begin{definition}[{\cite[p. 835]{bernstein2018scalar}}]
A matrix norm $\norm{\ \cdot\ }$ on $\mathbb{C}^{n\times m}$ is called \emph{unitarily invariant} if $\norm{UMV}=\norm{M}$ for all $M\in\mathbb{C}^{n\times m}$ and all unitary $U\in\mathbb{C}^{n\times n}$ and $V\in\mathbb{C}^{m\times m}$.
\end{definition}

For a matrix $M\in\mathbb{C}^{n\times m}$, its
\begin{enumerate}[label=-]
    \item Frobenius norm $\sqrt{\tr(MM^*)}=\sqrt{\sum_{i=1}^{\min\{n,m\}}\sigma_i^2(M)}$,
    \item Schatten $p$-norms $(\sum_{i=1}^{\min\{n,m\}}\sigma_i^p(M))^{1/p}$, $p\geq1$, and
    \item Ky Fan $k$-norms $\sum_{i=1}^k\sigma_i(M)$, $k\in[1,\min\{n,m\}]$,
\end{enumerate}
are all unitarily invariant \cite[p. 92]{bhatia2013matrix}. This class of norms has an interesting association with so-called symmetric gauge functions. 

\begin{definition}[{\cite[p. 51]{mirsky1960symmetric}}]
A function $g:\mathbb{R}^n\rightarrow \mathbb{R}_+$ is called a \emph{symmetric gauge function} if it satisfies the following properties\footnote{The concept of a gauge function, property (a), is due to Minkowski (see \cite[1.1(d)]{phelps1993convex}). The notion of symmetry, properties (b) and (c), was introduced by von Neumann in \cite{von1937some}.}:
\begin{enumerate}[label=\normalfont{(\alph*)}]
    \item $g$ is a norm, 
    \item $g(Px)=g(x)$ for all permutation matrices $P\in\mathbb{R}^{n\times n}$,
    \item If $s_i=\pm 1$, then $g(\begin{bmatrix}
    x_1 & \cdots & x_n
    \end{bmatrix}^\top)=g(\begin{bmatrix}
    s_1x_1 & \cdots & s_nx_n
    \end{bmatrix}^\top)$.
\end{enumerate}
\end{definition}

Having the above definition, it is not surprising that a symmetric gauge function of the vector of singular values of a matrix defines a unitarily invariant norm. Interestingly, von Neumann proved in \cite{von1937some} that the converse is also true.

\begin{proposition}[{\cite[Fact 11.9.59]{bernstein2018scalar}}] A norm $\norm{\ \cdot\ }$ on $\mathbb{C}^{n\times m}$ is unitarily invariant if and only if there exists a symmetric gauge function $g:\mathbb{R}^{\min\{n,m\}}\rightarrow \mathbb{R}_+$ such that $\norm{M}=g(\sigma(M))$ for all $M\in\mathbb{C}^{n\times m}$. 
\end{proposition}

We note that every unitarily invariant norm is associated with a \emph{unique} symmetric gauge function. The uniqueness is an immediate consequence of the symmetry. As an example, the symmetric gauge functions associated with the spectral norm and the Frobenius norm, respectively, are the infinity and the Euclidean vector norms. For a Schatten $p$-norm, $p\geq 1$, its associated symmetric gauge function is the vector $p$-norm. 

We know from Proposition \ref{prop:chebuniq-c} that for essentially strictly convex norms, the Chebyshev center of a compact set of matrices is unique. The following result shows that the essential strict convexity of a unitarily invariant norm is equivalent to that of its symmetric gauge function. 
\begin{proposition}[{\cite[Thm. 3.1]{ziketak1988characterization}}]
\label{prop:uin_strict}
A unitarily invariant norm is essentially strictly convex if and only if its associated symmetric gauge function is essentially strictly convex. 
\end{proposition}

For instance, based on Proposition \ref{prop:uin_strict}, the Frobenius norm is essentially strictly convex since the Euclidean norm is essentially strictly convex. Therefore, it follows from Proposition \ref{prop:chebuniq-c} that the Chebyshev center of a compact set with respect to the Frobenius norm is unique. However, this is not necessarily the case for the spectral norm.

\section{Main results}
\label{sec:IV}

The following theorem is the main result of this paper that provides closed-form expressions for a Chebyshev center, the Chebyshev radius, and the diameter of QMI-induced sets. 

\begin{theorem}
\label{th:1}
Let $\Pi\in\pmb{\Pi}_{q,p}^-$. Then, for any matrix norm
\begin{equation}
-\Pi_{22}^{-1} \Pi_{21}\in\cent\mathcal{Z}_{p}(\Pi)\hspace{0.5 cm} \text{and}\hspace{0.5 cm}  \diam\mathcal{Z}_{p}(\Pi)=2\rad \mathcal{Z}_{p}(\Pi).
\end{equation}
For a unitarily invariant norm with the associated symmetric gauge function $g$, we have
\begin{equation}
\rad\mathcal{Z}_{p}(\Pi)=g(\sigma_{[k]}((-\Pi_{22})^{-\frac{1}{2}})\circ \sigma_{[k]}((\Pi|\Pi_{22})^{\frac{1}{2}})),
\end{equation}
where $k=\min\{p,q\}$.
\end{theorem}

Recall from the example in Fig. \ref{fig:1} that the sets of Chebyshev centers for different norms do not necessarily intersect. However, Theorem \ref{th:1} reveals that for a QMI-induced set $\mathcal{Z}_p(\Pi)$ with $\Pi\in\pmb{\Pi}_{q,p}^-$,
\begin{equation}
\label{eq:chebcent}
\hat{Z}=-\Pi_{22}^{-1} \Pi_{21},
\end{equation}
is a \emph{common} Chebyshev center with respect to all matrix norms. In addition, recall from Proposition \ref{prop:chebuniq-c} that for essentially strictly convex norms, e.g., Frobenius norm, the set of Chebyshev centers is a singleton. This implies that \eqref{eq:chebcent} is the only common center among all (unitarily invariant) norms.

As an example, consider the ellipsoid $\tfrac{z_1^2}{a^2}+\tfrac{z_2^2}{b^2}+\tfrac{z_3^2}{c^2}\leq1$ shown in Fig. \ref{fig:2}. For this example, the Chebyshev radius presented in Theorem \ref{th:1} with respect to the maximum singular value is equal to the largest semi-axis $\max\{a,b,c\}$, that is the radius of the smallest sphere circumscribing the ellipsoid.

To prove Theorem \ref{th:1}, we need three auxiliary results. The first one deals with an upper bound on the unitarily invariant norm of a product of matrices. 

\begin{lemma}
\label{lem:maj}
Let $L\in\mathbb{R}^{l\times p}$, $S\in\mathbb{R}^{p\times q}$, $R\in\mathbb{R}^{q\times r}$, and $k\coloneqq\min\{l,r\}$. Also, let $g:\mathbb{R}^k\rightarrow \mathbb{R}_+$ be a symmetric gauge function. Then,
\begin{equation}
\label{eq:lem:maj-1}
g(\sigma_{[k]}(LSR))\leq g(\sigma_{[k]}(L)\circ \sigma_{[k]}(S)\circ \sigma_{[k]}(R)).
\end{equation}
\end{lemma}
\begin{proof}
The following majorization inequality holds for the singular values of the product of two matrices $X\in\mathbb{R}^{n_x\times m}$ and $Y\in\mathbb{R}^{m\times n_y}$:
\begin{equation}
\label{eq:lem:maj-2}
\sigma_{[s]}(XY)\prec_w \sigma_{[s]}(X)\circ \sigma_{[s]}(Y),
\end{equation}
for any $s\in\mathbb{N}$. A proof of inequality \eqref{eq:lem:maj-2} can be found in \cite[p. 342]{marshall2011inequalities} for square matrices, which can be straightforwardly extended to rectangular matrices by taking care of the dimension of the singular value vectors (see \cite[p. 299]{marshall2011inequalities} for further explanations). Next, to extend inequality \eqref{eq:lem:maj-2} to the product of three matrices, we consider the fact that if two vectors $x,y\in\mathbb{R}^n_+$ satisfy $x\prec_w y$, then for any $z\in\mathbb{R}_\downarrow^n$ we have $x\circ z\prec_w y\circ z$ (see \cite[Prob. 11.5.16]{bhatia2013matrix}). Hence, as $\sigma_{[k]}(R)\in\mathbb{R}_\downarrow^k$, inequality \eqref{eq:lem:maj-2} implies $\sigma_{[k]}(LS)\circ\sigma_{[k]}(R)\prec_w \sigma_{[k]}(L)\circ \sigma_{[k]}(S)\circ\sigma_{[k]}(R)$ by choosing $X=L$ and $Y=S$. Thus, by choosing $X=LS$ and $Y=R$ in \eqref{eq:lem:maj-2}, we have $\sigma_{[k]}(LSR)\prec_w \sigma_{[k]}(LS)\circ \sigma_{[k]}(R)$. Hence, we obtain
\begin{equation}
\label{eq:lem:maj-3}
\sigma_{[k]}(LSR)\prec_w \sigma_{[k]}(L)\circ \sigma_{[k]}(S)\circ \sigma_{[k]}(R).
\end{equation}
In addition, from Ky Fan's lemma, for vectors $x,y\in\mathbb{R}^n_+$, we have $g(x)\leq g(y)$ for all symmetric gauge functions if and only if $x\prec_w y$ (see \cite[Prop. 4.B.6]{marshall2011inequalities}). Therefore, inequality \eqref{eq:lem:maj-3} implies \eqref{eq:lem:maj-1} for all symmetric gauge functions. 
\end{proof}

The second auxiliary result deals with nonexpansive matrices.

\begin{lemma}
\label{lem:x-x}
Let $\mathcal{S}\coloneqq \{S\in\mathbb{R}^{n\times m}:SS^\top\leq I\}$. Then, $\mathcal{S}-\mathcal{S}=2\mathcal{S}$.
\end{lemma}
\begin{proof}
Let $Z\in2\mathcal{S}$. Then, $Z=2X$ for some $X\in\mathcal{S}$. Since $-X\in\mathcal{S}$ and $Z=X-(-X)$, we see that $Z\in\mathcal{S}-\mathcal{S}$. This proves that $2\mathcal{S}\subseteq \mathcal{S}-\mathcal{S}$. For the reverse inclusion, let $Z\in\mathcal{S}-\mathcal{S}$. Then, $Z=X-Y$ where $X,Y\in\mathcal{S}$. Note that
\begin{equation}
\label{eq:lem:x-x-3}
\begin{bmatrix}
I & X \\ X^\top & I
\end{bmatrix}\geq 0,\hspace{0.25 cm}\text{and}\hspace{0.25 cm} \begin{bmatrix}
I & -Y \\ -Y^\top & I
\end{bmatrix}\geq 0.
\end{equation}
Adding the matrices in \eqref{eq:lem:x-x-3} yields
\begin{equation}
\begin{bmatrix}
2I & Z \\ Z^\top & 2I
\end{bmatrix}\geq 0,
\end{equation}
equivalently, $ZZ^\top\leq 4I$. Then, $\tfrac{1}{2}Z\in\mathcal{S}$ and thus $Z\in2\mathcal{S}$. Therefore, $\mathcal{S}-\mathcal{S}\subseteq 2\mathcal{S}$, which completes the proof. 
\end{proof}

The third auxiliary result characterizes the diameter and Chebyshev radius and center for sets obtained from nonexpansive matrices. 

\begin{lemma}
\label{lem:cheball}
Let $\mathcal{S}\coloneqq \{S\in\mathbb{R}^{p\times q}:SS^\top\leq I\}$, $\mathcal{X}\coloneqq L\mathcal{S}R$ with $L\in\mathbb{R}^{l\times p}$ and $R\in\mathbb{R}^{q\times r}$, and $k\coloneqq\min \{r,l\}$. The following statements hold: 
\begin{enumerate}[label=\normalfont{(\alph*)},ref=\ref{lem:cheball}(\alph*)]
    \item\label{lem:cheball-c} For any norm, we have $0\in\cent\mathcal{X}$ and $\diam\mathcal{X}=2\rad\mathcal{X}$.
    \item\label{lem:cheball-a} For a unitarily invariant norm with associated symmetric gauge function $g$, we have $\rad\mathcal{X}=g(\sigma_{[k]}(L)\circ \sigma_{[k]}(R))$.
\end{enumerate}
\end{lemma}
\begin{proof}
Let $\rho=\max\{\norm{X}:X\in\mathcal{X}\}$. It follows from Lemma \ref{lem:x-x} that $\mathcal{X}-\mathcal{X}=2\mathcal{X}$ and from the definition of the diameter \eqref{eq:diam} that $\diam\mathcal{X}=\max\{\norm{\Delta}:\Delta\in2\mathcal{X}\}=2\rho$. On the one hand, Proposition \ref{prop:chebuniq-a} implies that $\rad\mathcal{X}\geq \rho$. On the other hand, we have
\begin{equation}
\rad\mathcal{X}=\min_{C_0\in\mathbb{R}^{p\times q}} \max_{X\in\mathcal{X}} \norm{C_0-X}\leq \max_{X\in\mathcal{X}} \norm{X}=\rho.
\end{equation}
Therefore, $\rad\mathcal{X}=\tfrac{1}{2}\diam\mathcal{X}=\rho$. Hence, it follows from the definition of the set of Chebyshev centers \eqref{eq:cent} that $0\in\cent\mathcal{X}$. This proves part (a). To prove part (b), assume that the norm is unitarily invariant with associated symmetric gauge function $g$. We show that \mbox{$\rho=g(\sigma_{[k]}(L)\circ \sigma_{[k]}(R))$}. To do so, observe that
\begin{equation}
\label{eq:lem:cheball*}
\norm{LSR}\leq g(\sigma_{[k]}(L)\circ \sigma_{[k]}(S)\circ \sigma_{[k]}(R)),
\end{equation}
due to Lemma \ref{lem:maj}. Since $\sigma_1(S)\leq 1$ whenever $SS^\top\leq I$, inequality \eqref{eq:lem:cheball*} yields 
\begin{equation}
\label{eq:lem:cheball**}
\rho\leq g(\sigma_{[k]}(L)\circ \sigma_{[k]}(R)).
\end{equation}
Now, we construct an element of $\mathcal{X}$ such that its norm is equal the upper bound in \eqref{eq:lem:cheball**}. For $M\in\{L,R\}$, let $M=U_M\Sigma_M V_M^\top$ be a singular value decomposition where $U_M$ and $V_M$ are orthogonal matrices, and $\Sigma_M$ is the matrix containing zero entries apart from its main diagonal which equals $\sigma(M)$. Moreover, let $\Sigma_S$ be a $p\times q$ matrix of the form $\begin{bmatrix}
I_k & 0 \\ 0 & 0
\end{bmatrix}$. Note that $S=V_L\Sigma_S U_R^\top$ satisfies $SS^\top\leq I$. Therefore, we have $LSR\in\mathcal{X}$ and
\begin{equation}
\norm{LSR}=\norm{\Sigma_L\Sigma_S\Sigma_R}=g(\sigma_{[k]}(L)\circ \sigma_{[k]}(R)).
\end{equation}
Together with \eqref{eq:lem:cheball**} this proves that $\rho= g(\sigma_{[k]}(L)\circ \sigma_{[k]}(R))$. 
\end{proof}

It is now straightforward to prove Theorem \ref{th:1} as follows. 

\begin{proof}[Proof of Theorem \ref{th:1}]
From Proposition \ref{prop:qmi-c}, we have
\begin{equation}
\mathcal{Z}_{p}(\Pi)=X_0+\mathcal{X}
\end{equation}
where $\mathcal{X}=\{LSR:SS^\top\leq I\}$, $X_0=-\Pi_{22}^{-1} \Pi_{21}$, $L=(-\Pi_{22})^{-\frac{1}{2}}$, and $R=(\Pi|\Pi_{22})^{\frac{1}{2}}$.
Based on Proposition \ref{prop:chebuniq-bc}, we have $\diam\mathcal{Z}_p(\Pi)=\diam\mathcal{X}$, $\rad\mathcal{Z}_p(\Pi)=\rad\mathcal{X}$, and \linebreak $\cent\mathcal{Z}_p(\Pi)=X_0+\cent\mathcal{X}$. Therefore, Theorem \ref{th:1}
now follows from Lemma \ref{lem:cheball}. 
\end{proof}

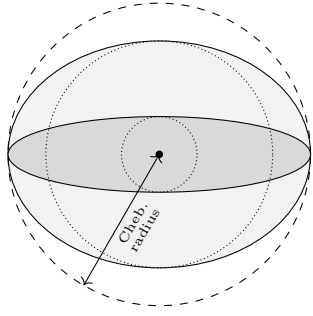
\begin{figure}[h]
    \centering
    \begin{tikzpicture}
    \draw[fill=gray!10] (0,0) ellipse (2cm and 1.5cm);
    \draw[fill=gray!30] (0,0) ellipse (2cm and 0.5cm);
    \draw[dashed] (0,0) ellipse (2cm and 2cm);
    \draw[densely dotted] (0,0) ellipse (0.5cm and 0.5cm);
    \draw[densely dotted] (0,0) ellipse (1.5cm and 1.5cm);
    \draw[rotate=-30,<->] (0,0)--(0,-2cm);
    \draw[fill=black,draw=none] (0,0) circle (0.5mm);
    \node[rotate=60,align=center] at (-0.35,-0.9) {\tiny Cheb.};
    \node[rotate=60,align=center] at (-0.2,-1) {\tiny radius};
    \end{tikzpicture}
    \caption{Two ellipses, shown by the dark and light gray areas, with the same Chebyshev radius but different volumes.}
    \label{fig:2pp}
\end{figure}

The Chebyshev centers and radii represent the best outer approximation for a set. However, two sets with the same Chebyshev radius may have different \emph{volumes}, e.g., see Fig. \ref{fig:2pp}. This motivates the study of the radius of a largest ball within the set as an inner approximation, which is another relevant concept in the analysis of QMI-induced sets. To formulate this, let the ball centered at $X_0\in\mathbb{R}^{p\times q}$ with radius $\rho\geq0$ be defined as
\begin{equation}
\mathcal{B}(X_0,\rho)\coloneqq\{X:\norm{X-X_0}\leq \rho\}.
\end{equation}
Given $\Pi\in\pmb{\Pi}_{q,p}^-$, the Chebyshev radius of $\mathcal{Z}_p(\Pi)$ is the radius of the smallest ball containing the set $\mathcal{Z}_p(\Pi)$, that is, 
\begin{equation}
\rad\mathcal{Z}_p(\Pi)=\min\{\rho:\mathcal{Z}_p(\Pi)\subseteq\mathcal{B}(\hat{Z},\rho)\},
\end{equation}
where $\hat{Z}$ is the common Chebyshev center defined in \eqref{eq:chebcent}. Now, aside from an outer ball with radius $\rho_1$ where $\mathcal{Z}_p(\Pi)\subseteq\mathcal{B}(\hat{Z},\rho_1)$, one may consider an inner ball with radius $\rho_2$ such that $\mathcal{B}(\hat{Z},\rho_2)\subseteq\mathcal{Z}_p(\Pi)$. The radius of such an inner ball is zero if the interior of the QMI-induced set in $\mathbb{R}^{p\times q}$ is empty. This is the case if and only if $\Pi|\Pi_{22}$ has a zero eigenvalue (see \cite[Thm. 3.2(c)]{van2023quadratic}). However, in that case, one may study the radius of a largest lower-dimensional ball within the set. Hence, we define
\begin{equation}
\rad_\text{in}\mathcal{Z}_p(\Pi)\coloneqq \max\{\rho:\mathcal{B}(\hat{Z},\rho)\cap\aff\mathcal{Z}_p(\Pi)\subseteq\mathcal{Z}_p(\Pi)\},
\end{equation}
which is characterized by the following theorem.

\begin{theorem}
\label{th:2}
Let $\Pi\in\pmb{\Pi}_{q,p}^-$. For a unitarily invariant norm with associated symmetric gauge function $g$, we have 
\begin{equation}
\rad_{\textup{in}}\mathcal{Z}_p(\Pi)=\sigma_*((\Pi|\Pi_{22})^{\frac{1}{2}})\sigma_p((-\Pi_{22})^{-\frac{1}{2}})g(e_1).
\end{equation}
\end{theorem}

Unitarily invariant norms are called \emph{normalized} if their associated symmetric gauge function satisfies $g(e_1)=1$ \cite{hiai2002inequalities}. All Schatten $p$-norms and Ky Fan $k$-norms are normalized. However, in general, this is not the case as one may consider a weighted norm such that $g(e_1)\neq1$. 

Once more, consider the example of the ellipsoid $\tfrac{z_1^2}{a^2}+\tfrac{z_2^2}{b^2}+\tfrac{z_3^2}{c^2}\leq1$ shown in Fig. \ref{fig:2p}. For this example, the radius of the largest inner ball presented in Theorem \ref{th:2} with respect to the maximum singular value is equal to the smallest semi-axis $\min\{a,b,c\}$, which is equal to the radius of the largest sphere inscribed in the ellipsoid. Now, consider the example of an ellipse in $\mathbb{R}^3$ shown in Fig. \ref{fig:3} defined by $\tfrac{z_2^2}{b^2}+\tfrac{z_3^2}{c^2}\leq1$ and $z_1=0$. For this example, the radius of the largest sphere within the ellipse is zero. However, the radius of the inner ball presented in Theorem \ref{th:2} is the radius of the largest circle within the ellipse, which is equal to the smallest semi-axis $\min\{b,c\}$. 

\begin{figure}[h]
\centering
\begin{subfigure}[b]{0.49\textwidth}
    \centering
    \begin{tikzpicture}
    \draw[rotate=-10,fill=gray!10] (0,0) ellipse (3cm and 1cm);
    \draw[fill=gray!30] (0,0) ellipse (1cm and 1cm);
    \draw[dashed,rotate=-10] (3cm,0)  arc (0:180:3cm and 0.5cm);
    \draw[dashed,rotate=-10] (1cm,0)  arc (0:180:1cm and 0.5cm);
    \draw[rotate=-10] (-3cm,0)  arc (180:360:3cm and 0.5cm);
    \draw[rotate=-10] (-1cm,0)  arc (180:360:1cm and 0.5cm);
    \draw[rotate=-10] (0,1cm) arc (90:270:0.5cm and 1cm);
    \draw[dashed,rotate=-10] (0,-1cm) arc (-90:90:0.5cm and 1cm);
    \draw[rotate=-10,densely dotted] (0,0)--(3cm,0);
    \draw[rotate=-10,->] (0,0)--(0.25cm,0);
    \draw[rotate=-10,densely dotted] (0,0)--(0,1cm);
    \draw[rotate=-10,->] (0,0)--(0,0.25cm);
    \draw[densely dotted] (0,0)--(-0.5cm,-0.4cm);
    \draw[->] (0,0)--(-0.1666cm,-0.1333cm);
    \draw[fill=black,draw=none] (0,0) circle (0.5mm);
    \node[align=center] at (0.2,-0.2) {\footnotesize $z_3$};
    \node[align=center] at (-0.3,0) {\footnotesize $z_2$};
    \node[align=center] at (0.25,0.2) {\footnotesize $z_1$};
    \end{tikzpicture}
    \caption{The largest sphere within an ellipsoid.}
    \label{fig:2p}
    \end{subfigure}
    \hfill
    \begin{subfigure}[b]{0.49\textwidth}
    \centering
    \begin{tikzpicture}
    \draw[fill=gray!10,rotate=-10] (3cm,0)  arc (0:180:3cm and 0.5cm);
    \draw[fill=gray!10,rotate=-10] (-3cm,0)  arc (180:360:3cm and 0.5cm);
    \draw[fill=gray!30,rotate=-10] (1cm,0)  arc (0:180:1cm and 0.5cm);
    \draw[fill=gray!30,rotate=-10] (-1cm,0)  arc (180:360:1cm and 0.5cm);
    \draw[rotate=-10,densely dotted] (0,0)--(3cm,0);
    \draw[rotate=-10,->] (0,0)--(0.5cm,0);
    \draw[rotate=-10,->] (0,0)--(0,1cm);
    \draw[densely dotted] (0,0)--(-0.5cm,-0.4cm);
    \draw[->] (0,0)--(-0.3332cm,-0.2666cm);
    \draw[fill=black,draw=none] (0,0) circle (0.5mm);
    \node[align=center] at (0.2,-0.2) {\footnotesize $z_3$};
    \node[align=center] at (-0.3,0) {\footnotesize $z_2$};
    \node[align=center] at (0.4,0.7) {\footnotesize $z_1$};
    \end{tikzpicture}
    \caption{The largest circle within an ellipse in $\mathbb{R}^3$.}
    \label{fig:3}
    \end{subfigure}
    \caption{Examples of the inner balls for QMI-induced sets.}
\end{figure}
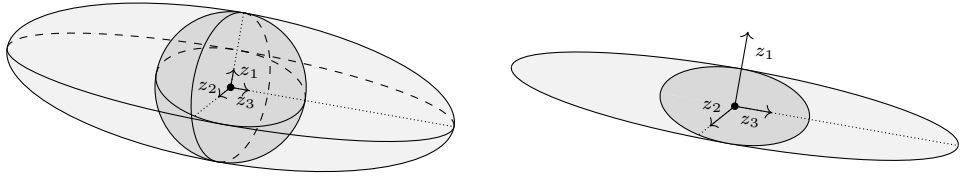

To prove Theorem \ref{th:2}, we need the following lemma. 

\begin{lemma}
\label{lem:ball_in}
Let $\rho\geq 0$, $\mathcal{S}\coloneqq \{S\in\mathbb{R}^{p\times q}:SS^\top\leq I\}$, and $\mathcal{X}\coloneqq L\mathcal{S}R$ with $L\in\mathbb{R}^{l\times p}$ and $R\in\mathbb{R}^{q\times r}$. For a unitarily invariant norm with associated symmetric gauge function $g$, we have
\begin{equation}
\label{eq:lem:ball_in}
\mathcal{B}(0,\rho)\cap\aff\mathcal{X}\subseteq\mathcal{X} 
\end{equation}
if and only if $\rho\leq \sigma_*(L)\sigma_*(R)g(e_{1})$. 
\end{lemma}
\begin{proof}
It is easy to see that if $L=0$ or $R=0$, then $\mathcal{X}=\{0\}$ is a singleton and the claim holds. Hence, in what follows, we assume that $L$ and $R$ are both nonzero. For the ``if'' part, assume that $\rho\leq\sigma_*(L)\sigma_*(R)g(e_1)$. Let $X\in\mathcal{B}(0,\rho)\cap\aff\mathcal{X}$. Since $\aff \mathcal{X}=\{LTR:T\in\mathbb{R}^{p\times q}\}$, we have $X=\rho LTR$ for some $T\in\mathbb{R}^{p\times q}$ satisfying $\norm{LTR}\leq 1$. Take $S=\rho L^{\dagger}LTRR^{\dagger}$, and observe that 
\begin{equation}
\sigma_1(S)=\rho\sigma_1(L^{\dagger}LTRR^{\dagger})\leq \sigma_*(L)\sigma_*(R)g(e_1) \sigma_1(L^{\dagger})\sigma_1(LTR)\sigma_1(R^{\dagger}).
\end{equation}
Since $\sigma_1(L^\dagger)=1/\sigma_*(L)$ and $\sigma_1(R^\dagger)=1/\sigma_*(R)$, we have $\sigma_1(S)\leq g(e_1) \sigma_1(LTR)$. Since $\sigma(LTR)\succ_w\sigma_1(LTR)e_1$, it follows from Ky Fan's lemma \cite[Prop. 4.B.6]{marshall2011inequalities} that
\begin{equation}
g(\sigma(LTR))\geq g(\sigma_1(LTR)e_1)=\sigma_1(LTR)g(e_1).
\end{equation}
Thus, $\sigma_1(LTR)\leq g(\sigma(LTR))/g(e_1)=\norm{LTR}/g(e_1) \leq 1/g(e_1)$. This implies that \mbox{$\sigma_1(S)\leq 1$}. Therefore, since $X=LSR$ for some $S\in\mathcal{S}$, we have $X\in\mathcal{X}$. To prove the ``only if'' part, assume that \eqref{eq:lem:ball_in} holds. This implies that for every $T\in\mathbb{R}^{p\times q}$ with $\norm{LTR}\leq 1$ we have $\rho LTR=LSR$ for some $S\in\mathcal{S}$. Based on the singular value decomposition, consider a decomposition of $L$ ($R$) as $U_L\Sigma_LV_L^\top$ ($U_R\Sigma_RV_R^\top$) such that $U_L$ and $V_L$ ($U_R$ and $V_R$) are orthogonal matrices and $\Sigma_L$ ($\Sigma_R$) has $\sigma_*(L)$ ($\sigma_*(R)$) on its $(1,1)$-th entry. Using this decomposition, we multiply $\rho LTR=LSR$ from left and right by $U_L^\top$ and $V_R$, respectively, to have the equivalent expression
\begin{equation}
\label{eq:lem:ball-in-pf}
\rho \Sigma_LV_L^\top TU_R\Sigma_R =\Sigma_LV_L^\top S U_R\Sigma_R.
\end{equation}
Take $T=V_L\Sigma_T U_R^\top$ where $\Sigma_T$ has all its entries equal to zero but the $(1,1)$-entry, which is equal to $1/(\sigma_*(L)\sigma_*(R)g(e_1))$. With this choice, the condition $\norm{LTR}=1$ is satisfied since 
\begin{equation}
\norm{LTR}=\norm{\Sigma_L \Sigma_T\Sigma_R}=g((1/g(e_1))e_1)=(1/g(e_1))g(e_1)=1.
\end{equation}
Substitute 
$T=V_L\Sigma_T U_R^\top$ in \eqref{eq:lem:ball-in-pf} to have
\begin{equation}
\label{eq:lem:ball-in-pf2}
\rho \Sigma_L\Sigma_T\Sigma_R =\Sigma_LV_L^\top S U_R\Sigma_R.
\end{equation}
The $(1,1)$-th entry of \eqref{eq:lem:ball-in-pf2} reads 
$\rho/g(e_1) =\sigma_*(L)\sigma_*(R)\tilde{s}$, where $\tilde{s}$ is the $(1,1)$-entry of $V_L^\top S U_R$. Since matrix entries are upper bounded by the largest singular value (see \cite[Fact 11.9.23(xii)]{bernstein2018scalar}), we have $\tilde{s}\leq \sigma_1(V_L^\top S U_R)=\sigma_1(S)\leq 1$. Therefore, we have \mbox{$\rho/g(e_1) \leq\sigma_*(L)\sigma_*(R)$}, which completes the proof. 
\end{proof}

It is now straightforward to prove Theorem \ref{th:2} as follows.

\begin{proof}[Proof of Theorem \ref{th:2}]
From Proposition \ref{prop:qmi-c} we have $\mathcal{Z}_{p}(\Pi)=X_0+\mathcal{X}$ where $X_0=-\Pi_{22}^{-1} \Pi_{21}$ and $\mathcal{X}=\{LSR:SS^\top\leq I\}$ with $L=(-\Pi_{22})^{-\frac{1}{2}}$ and $R=(\Pi|\Pi_{22})^{\frac{1}{2}}$. We also verify that $\mathcal{B}(0,\rho)\cap\aff\mathcal{X}\subseteq \mathcal{X}$ if and only if $\mathcal{B}(X_0,\rho)\cap\aff(\mathcal{X}+X_0)\subseteq X_0+\mathcal{X}$.  Therefore, Theorem \ref{th:2}
now follows from Lemma \ref{lem:ball_in}. 
\end{proof}

Using Lemmas \ref{lem:cheball} and \ref{lem:ball_in}, the results of Theorems \ref{th:1} and \ref{th:2} can be extended to a more general class of matrix sets, which is presented in the following proposition.

\begin{proposition}
\label{prop:3.7}
Let $\Pi\in\pmb{\Pi}_{q,p}^-$, $\Theta_{L}\in\mathbb{R}^{l\times p}$, and $\Theta_{R}\in\mathbb{R}^{q\times r}$. Then, for any norm we have
\begin{equation}
\label{eq:prop:3.7-1}
-\Theta_L\Pi_{22}^{-1} \Pi_{21}\Theta_R\in\cent(\Theta_L\mathcal{Z}_{p}(\Pi)\Theta_R),
\end{equation}
and
\begin{equation}
\label{eq:prop:3.7-2}
\diam(\Theta_{L}\mathcal{Z}_{p}(\Pi)\Theta_{R})=2\rad (\Theta_{L}\mathcal{Z}_{p}(\Pi)\Theta_{R}).
\end{equation}
For a unitarily invariant norm with associated symmetric gauge function $g$, we have
\begin{equation}
\label{eq:prop:3.7-3}
\rad (\Theta_{L}\mathcal{Z}_{p}(\Pi)\Theta_{R})=g(\sigma_{[k]}(\Theta_{L}(-\Pi_{22})^{-\frac{1}{2}})\circ \sigma_{[k]}((\Pi|\Pi_{22})^{\frac{1}{2}}\Theta_{R})),
\end{equation} 
where $k\coloneqq\min\{l,r\}$, and
\begin{equation}
\label{eq:prop:3.7-4}
\rad_\textup{in} (\Theta_{L}\mathcal{Z}_{p}(\Pi)\Theta_{R})=\sigma_{*}(\Theta_{L}(-\Pi_{22})^{-\frac{1}{2}}) \sigma_{*}((\Pi|\Pi_{22})^{\frac{1}{2}}\Theta_{R})g(e_1).
\end{equation} 
\end{proposition}
\begin{proof}
Based on Proposition \ref{prop:qmi-c}, we have 
\begin{equation}
\Theta_{L}\mathcal{Z}_{p}(\Pi)\Theta_{R}=X_0+\mathcal{X},
\end{equation}
where $X_0=-\Theta_{L}\Pi_{22}^{-1} \Pi_{21}\Theta_{R}$ and $\mathcal{X}=L\mathcal{S}R$ with $L=\Theta_{L}(-\Pi_{22})^{-\frac{1}{2}}$ and $R=(\Pi|\Pi_{22})^{\frac{1}{2}}\Theta_{R}$. Now, \eqref{eq:prop:3.7-1}, \eqref{eq:prop:3.7-2}, and \eqref{eq:prop:3.7-3} follow from Lemma \ref{lem:cheball}. Using a similar argument as in the proof of Theorem \ref{th:2}, the proof of \eqref{eq:prop:3.7-4} follows from Lemma \ref{lem:ball_in}. 
\end{proof}

\section{Applications to data-driven modeling and control}
\label{sec:III}

In this section, we discuss the applications of the presented results in data-driven modeling and control using a set-membership approach. Consider multiple-input multiple-output (MIMO) autoregressive models with exogenous input (ARX models) given by
\begin{equation}
\label{eq:2}
y(t+L)=\sum_{i=0}^{L-1}P_iy(t+i)+\sum_{j=0}^{M}Q_ju(t+j)+w(t),
\end{equation}
where $y(t)\in\mathbb{R}^p$ is the output, $u(t)\in\mathbb{R}^m$ is the input, $w(t)\in\mathbb{R}^p$ is the process noise, the integers $L$ and $M$ are known, and matrices $P_i\in\mathbb{R}^{p\times p}$, $i=0,\ldots,L-1$, and $Q_j\in\mathbb{R}^{p\times m}$, $j=0,\ldots,M$, are unknowns. We assume that the system is causal, i.e., $M\leq L$. Note that input-state models can be captured by choosing $L=1$ and $M=0$. We refer to \eqref{eq:2} as the true system and we define
\begin{equation}
P_\textup{true}\coloneqq\begin{bmatrix}
P_0 & \cdots & P_{L-1}
\end{bmatrix}\ \ \text{and}\ \  Q_\textup{true}\coloneqq\begin{bmatrix}
Q_0 & \cdots & Q_{M}
\end{bmatrix}.
\end{equation}

We collect the input-output data 
\begin{equation}
\label{eq:i-o-data}
u(0), \ldots, u(T+M-L),\hspace{0.5 cm} y(0), \ldots, y(T),
\end{equation}
from \eqref{eq:2}, where $T\geq L$. The process noise $w(t)$ that affects the dynamics during this experiment is unknown, but the matrix
\begin{equation*}
W_-\coloneqq\begin{bmatrix}
w(0) & \cdots & w(T-L)
\end{bmatrix}
\end{equation*}
is assumed to satisfy an energy bound
\begin{equation}
\label{eq:ass1-1}
\begin{bmatrix}
I \\ W_-^\top
\end{bmatrix}^\top \begin{bmatrix}
\Phi_{11} & \Phi_{12} \\
\Phi_{12}^\top & \Phi_{22}
\end{bmatrix} \begin{bmatrix}
I \\ W_-^\top
\end{bmatrix}\geq 0,
\end{equation}
i.e., $W_-^\top\in \mathcal{Z}_{T-L+1}(\Phi)$, for a given matrix $\Phi\in\pmb{\Pi}_{p,T-L+1}^-$. The reader can refer to \mbox{\cite[p. 4]{van2023quadratic}} for the special cases that can be captured by this noise model. 

Define the matrices
\begin{equation*}
\begin{split}
U_-&\coloneqq \begin{bmatrix}
u(0) & u(1) & \cdots & u(T-L) \\ 
\vdots & \vdots & \ddots & \vdots \\
u(M) & u(M+1) & \cdots & u(T+M-L)
\end{bmatrix}, \\
Y_-&\coloneqq \begin{bmatrix}
y(0) & y(1) & \cdots & y(T-L) \\
\vdots & \vdots & \ddots & \vdots \\
y(L-1) & y(L) & \cdots & y(T-1)
\end{bmatrix},\ \text{and} \\
Y_+&\coloneqq \begin{bmatrix}
y(L) & y(L+1) & \cdots & y(T)\end{bmatrix}.
\end{split}
\end{equation*}
A pair of real matrices $(P,Q)$ satisfying
\begin{equation}
\label{eq:7}
Y_+=PY_-+QU_-+W_-
\end{equation}
for some $W_-^\top\in \mathcal{Z}_{T-L+1}(\Phi)$ is called a \emph{data-consistent system}. We define the set of all data-consistent systems as
\begin{equation}
\label{eq:8}
\Sigma\coloneqq \left\{(P,Q):(Y_+-PY_--QU_-)^\top\in \mathcal{Z}_{T-L+1}(\Phi)\right\},
\end{equation}
which is the set of all $(P,Q)$ that could have generated the available data for some noise sequence agreeing with \eqref{eq:ass1-1}. 

One can verify that the set of data-consistent systems is equivalent to a QMI-induced set (cf. \cite[Lem. 4]{van2020noisy} and \cite[Sec. IV]{van2023behavioral}), that is,
\begin{equation}
\label{eq:Sigma->QMI}
(P,Q)\in\Sigma\ \textup{ if and only if }\ \begin{bmatrix}
P & Q
\end{bmatrix}^\top\in\mathcal{Z}_{s}(N),
\end{equation}
where $s\coloneqq Lp+(M+1)m$ denotes the number of columns of $\begin{bmatrix}
    Y_-^\top & U_-^\top
    \end{bmatrix}$ and 
\begin{equation}
\label{eq:10}
N\coloneqq 
\begin{bmatrix}
I & Y_+ \\ 
0 & -Y_- \\
0 & -U_-
\end{bmatrix}\begin{bmatrix}
\Phi_{11} & \Phi_{12} \\
\Phi_{12}^\top & \Phi_{22}
\end{bmatrix}\begin{bmatrix}
I & Y_+ \\ 
0 & -Y_- \\
0 & -U_-
\end{bmatrix}^\top .
\end{equation}
Furthermore, it can be observed that matrix $N$ in \eqref{eq:10} satisfies $N\in\pmb{\Pi}_{p,s}$. Hence, based on Proposition \ref{prop:qmi}, we have the following facts.

\begin{proposition}
\label{prop:xu}
The following statements hold: 
\begin{enumerate}[label=\normalfont{(\alph*)},ref=\ref{prop:xu}(\alph*)]
    \item\label{prop:xu-a} $\Sigma$ is nonempty.
    \item\label{prop:xu-b} $\Sigma$ is bounded if and only if $\rank\begin{bmatrix}
    Y_-^\top & U_-^\top
    \end{bmatrix}=s$.
    \item\label{prop:xu-c} Assume that $\Sigma$ is bounded. Then, $\Sigma$ is a singleton if and only if $N|N_{22}=0$. In this case, $\Sigma=\{(P_{\textup{true}},Q_{\textup{true}})\}$ and $\begin{bmatrix} P_{\textup{true}} & Q_{\textup{true}} \end{bmatrix}=-N_{12}N_{22}^{-1}$.
\end{enumerate}
\end{proposition}

In the noiseless case, $\Phi_{11}=0$, $\Phi_{12}=\Phi_{21}^\top=0$, and $\Phi_{22}=-I$, the condition of Proposition \ref{prop:xu-c}, $N|N_{22}=0$, is always satisfied. In the presence of noise, a necessary condition for unique identification is that the noise sequence is such that the left-hand side of \eqref{eq:ass1-1} is equal to zero, i.e., the noise is of bound-exploring type (cf. \cite[Ass. 5]{lauricella2020set}). However, in general, $\mathcal{Z}_{s}(N)$ is not a singleton and $(P_{\textup{true}},Q_{\textup{true}})$ cannot be uniquely recovered from the data. Nevertheless, in this scenario, we can still formalize system identification as the problem of finding a Chebyshev center, the Chebyshev radius, and the diameter of $\mathcal{Z}_{s}(N)$. 

\subsection{Best worst-case estimation}
\label{subsec:UniEs}

Let $(\hat{P},\hat{Q})$ be an estimation for the true system. The worst-case error of this estimation with respect to a matrix norm $\norm{\ \cdot\ }$ is defined as
\begin{equation}
\max_{(P,Q)\in\Sigma}\norm{\begin{bmatrix}
\hat{P} & \hat{Q}
\end{bmatrix}-\begin{bmatrix}
P & Q
\end{bmatrix}}.
\end{equation}
A best worst-case estimation for the true system is the one for which the worst-case error is minimized. Based on Theorem \ref{th:1}, if the data satisfy $\rank\begin{bmatrix}
    Y_-^\top & U_-^\top
    \end{bmatrix}=s$, then such an estimation is given by the following Chebyshev center:
\begin{equation}
\label{eq:chebcent_app}
\begin{bmatrix} \hat{P} & \hat{Q} \end{bmatrix}\coloneqq-N_{12}N_{22}^{-1}
=\begin{bmatrix}
I \\ Y_+^\top
\end{bmatrix}^\top\begin{bmatrix}
\Phi_{12} \\ \Phi_{22}
\end{bmatrix}\begin{bmatrix}
Y_- \\ U_-
\end{bmatrix}^\top \left(\begin{bmatrix}
Y_- \\ U_-
\end{bmatrix}\Phi_{22}\begin{bmatrix}
Y_- \\ U_-
\end{bmatrix}^\top\right)^{-1}.
\end{equation}
We note that \eqref{eq:chebcent_app} is a data-consistent system, $(\hat{P},\hat{Q})\in\Sigma$, and the only common center for all matrix norms. Therefore, $(\hat{P},\hat{Q})$ given by \eqref{eq:chebcent_app} is the \emph{only} estimation that minimizes the worst-case error with respect to \emph{all} matrix norms. For a unitarily invariant norm associated with a symmetric gauge function $g$, the worst-case error of this estimation is
\begin{equation}
\label{eq:Chebrad_app}
\rad\mathcal{Z}_s(N)=g(\sigma_{[p]}((-N_{22})^{-\frac{1}{2}})\circ \sigma((N|N_{22})^{\frac{1}{2}})).
\end{equation}
For this estimation, assuming that the norm is normalized, the radius of the largest inner ball is
\begin{equation}
\label{eq:inner_app}
\rad_{\textup{in}}\mathcal{Z}_s(N)=\sigma_*((N|N_{22})^{\frac{1}{2}})\sigma_s((-N_{22})^{-\frac{1}{2}}).
\end{equation}

\subsection{Least-squares solution as a Chebyshev center}
\label{subsec:LS}

It turns out that the common Chebyshev center \eqref{eq:chebcent_app} coincides with the solution of a matrix least-squares problem. Here, we denote the Frobenius norm by $\norm{\ \cdot\ }_\textup{F}$.
\begin{proposition}
Suppose that $\rank\begin{bmatrix}
Y_-^\top & U_-^\top
\end{bmatrix}=s$. Then, we have
\begin{equation}
\label{eq:GLS}
(\hat{P},\hat{Q})=\arg\min_{(P,Q)} \norm{(PY_-+QU_--Y_+-\Phi_{12}\Phi_{22}^{-1}) (-\Phi_{22})^{\frac{1}{2}}}_{\textup{F}}.
\end{equation}
\end{proposition}
\begin{proof}
For given $A\in\mathbb{R}^{s\times l}$ of full row rank and $B\in\mathbb{R}^{p\times l}$, we recall that the least-squares problem
\begin{equation}
\label{eq:LS}
\arg\min_{X\in\mathbb{R}^{p\times s}} |XA-B|_\text{F}
\end{equation}
has a unique solution $\hat{X}=BA^\dagger=BA^\top(AA^\top)^{-1}$, see \cite[Eq. (12.12)]{boyd2018introduction}. Now, to show that \eqref{eq:GLS} holds, we take
\begin{equation}
X=\begin{bmatrix}
P & Q
\end{bmatrix},\  A=\begin{bmatrix}
Y_- \\ U_-
\end{bmatrix}(-\Phi_{22})^{\frac{1}{2}},\ \text{and}\ B=(Y_++\Phi_{12}\Phi_{22}^{-1})(-\Phi_{22})^{\frac{1}{2}}.
\end{equation}
\end{proof}

We note that \eqref{eq:GLS} can be seen as a generalized least-squares problem with the bias term $\Phi_{12}\Phi_{22}^{-1}$ taken into account and $\Phi_{22}$ acting as the precision matrix. If \mbox{$\Phi_{22}=-I$} and $\Phi_{12}=0$, then \eqref{eq:GLS} is the solution to the ordinary least-squares problem \mbox{(cf. \cite{fogel1979system,bertsekas1971recursive,milanese1995properties})}.

\subsection{Prior knowledge}

The set of data-consistent systems $\Sigma$ is merely constructed based on the collected data and the known model class. However, incorporating further prior knowledge may result in a smaller estimation error. Let us introduce the set $\Sigma_{\textup{pk}}$ representing any prior knowledge of the true system coefficients. Given $\Sigma$ and $\Sigma_{\textup{pk}}$, the set of data-consistent systems complying with the prior knowledge is $\Sigma\cap\Sigma_{\textup{pk}}$.  In the case that $\Sigma\cap\Sigma_{\textup{pk}}$ is bounded and QMI-induced, the result of Theorem \ref{th:1} can be extended accordingly. Note that prior knowledge could also reduce the rank condition on the data matrix $\begin{bmatrix}
Y_-^\top & U_-^\top
\end{bmatrix}$ as $\Sigma\cap\Sigma_{\textup{pk}}$ may be bounded even if $\Sigma$ is unbounded (see Proposition \ref{prop:xu-b}). 

As an example, suppose that for known matrices $\Omega_2\in\mathbb{R}^{h\times s}$ and $\Omega_1\in\mathbb{R}^{p\times s}$ we have
\begin{equation}
\label{eq:pk1}
\Sigma_{\textup{pk}}\coloneqq\left\{(P,Q):\begin{bmatrix}
P & Q
\end{bmatrix}=H\Omega_2+\Omega_1\ \text{for some}\ H\in\mathbb{R}^{p\times h}\right\}.
\end{equation}
This type of constraint can be used when some columns of $\begin{bmatrix}
P_\textup{true} & Q_\textup{true}
\end{bmatrix}$ are known a priori so that the unknown terms are included in $H$. It follows now from \eqref{eq:Sigma->QMI} and \eqref{eq:pk1} that $(P,Q)\in\Sigma\cap \Sigma_{\textup{pk}}$ is equivalent to $(H\Omega_2+\Omega_1)^\top\in\mathcal{Z}_s(N)$. This inclusion can be written as
\begin{equation}
\label{eq:HinZ}
H^\top\in\mathcal{Z}_{h}(N_{\textup{pk}})
\end{equation}
with $N_{\textup{pk}}$ defined as
\begin{equation}
N_{\textup{pk}}\coloneqq \begin{bmatrix}
I & \Omega_1 \\
0 & \Omega_2
\end{bmatrix}N \begin{bmatrix}
I & \Omega_1 \\
0 & \Omega_2
\end{bmatrix}^\top.
\end{equation}
Therefore, $(P,Q)\in\Sigma\cap \Sigma_{\textup{pk}}$ if and only if \eqref{eq:HinZ} holds. In this case, one can observe that for $\Sigma\cap \Sigma_{\textup{pk}}$ to be bounded, we need $\begin{bmatrix}
Y_-^\top & U_-^\top
\end{bmatrix}\Omega_2^\top$ to be full column rank (cf. Proposition \ref{prop:xu-b}). Having a bounded $\Sigma\cap \Sigma_{\textup{pk}}$, the results of Theorem \ref{th:1} hold for the set of data-consistent systems incorporating prior knowledge \eqref{eq:pk1} by replacing $N$ with $N_{\textup{pk}}$.  

\subsection{The choice of norm}

Theorem \ref{th:1} is applicable to arbitrary unitarily invariant norms. Depending on the application, certain norms could be more relevant than others. For instance, suppose that the goal is to design a robust controller for the set of data-consistent systems. In this case, considering a Chebyshev center $(\hat{P},\hat{Q})$ to be the nominal system with the bound on the uncertainty $(\Delta_P,\Delta_Q)$ given by the Chebyshev radius with respect to the spectral norm, one may check the feasibility to design a robust controller for all systems in the form of $(\hat{P}+\Delta_P,\hat{Q}+\Delta_Q)$. This can be accomplished, e.g., using the results in \cite{khargonekar1990robust}. Such a parametrization can be obtained from the following proposition. Here, $\lambda_\textup{min}(M)$ and $\lambda_\textup{max}(M)$ denote the smallest and largest eigenvalues of a square matrix $M$, respectively. 

\begin{proposition}
\label{prop:spec_norm_rad}
Suppose that $N_{22}<0$. Then, we have
\begin{equation}
\label{eq:cor:sv-1}
\sigma_{1}\left(\begin{bmatrix}
\Delta_P & \Delta_Q
\end{bmatrix}\right)\leq \rho,
\end{equation}
for all $(\Delta_P,\Delta_Q)\in\Sigma-(\hat{P},\hat{Q})$ if and only if
\begin{equation}
\label{eq:cor:sv-2}
\rho \geq\sqrt{\frac{\lambda_{\textup{max}}(N|N_{22})}{\lambda_{\textup{min}}(-N_{22})}}.
\end{equation}
\end{proposition}
\begin{proof}
It suffices to show that, with respect to the spectral norm, we have
\begin{equation}
\label{eq:rad_pf_0}
\rad\mathcal{Z}_s(N)=\sqrt{\frac{\lambda_{\textup{max}}(N|N_{22})}{\lambda_{\textup{min}}(-N_{22})}}. 
\end{equation}
Recall that the symmetric gauge function associated with the spectral norm is the infinity norm on $\mathbb{R}^n$. Thus, it follows from \eqref{eq:Chebrad_app} that
\begin{equation}
\label{eq:rad_pf_1}
\rad\mathcal{Z}_s(N)=\sigma_{1}((-N_{22})^{-\frac{1}{2}}) \sigma_1((N|N_{22})^{\frac{1}{2}})).
\end{equation}
Since $N|N_{22}\geq 0$ and $N_{22}<0$, we have
\begin{equation}
\label{eq:rad_pf_2}
\begin{split}
\sigma_1((N|N_{22})^{\frac{1}{2}})&=\lambda_\text{max}((N|N_{22})^{\frac{1}{2}})=\sqrt{\lambda_\text{max}(N|N_{22})}\ \text{ and}\\ 
\sigma_{1}((-N_{22})^{-\frac{1}{2}})&=\lambda_\text{max}((-N_{22})^{-\frac{1}{2}})=\sqrt{\lambda_\text{max}((-N_{22})^{-1})}=\frac{1}{\sqrt{\lambda_\text{min}(-N_{22})}}.
\end{split}
\end{equation}
Now, \eqref{eq:rad_pf_0} follows from \eqref{eq:rad_pf_1} and \eqref{eq:rad_pf_2}. 
\end{proof}

A similar observation to Proposition \ref{prop:spec_norm_rad} was also made in \cite{eising2023data}. For the special case $L=1$, $M=0$, $\Phi_{22}=-I$, and $\Phi_{12}=\Phi_{21}^\top=0$, it is proven in \mbox{\cite[Lem. IV.5]{eising2023data}} that the set of data-consistent systems $\mathcal{Z}_s(N)$ is contained in a ball with radius $\sqrt{\lambda_{\textup{max}}(N | N_{22}) / \lambda_{\textup{min}}(-N_{22})}$ centered at $-N_{22}^{-1}N_{21}$. Interestingly, our result shows that this ball is the one with \emph{smallest} radius containing $\mathcal{Z}_s(N)$.

The majority of the literature on robust control deals with uncertainties that are bounded with respect to the spectral norm. Nevertheless, it was argued in \cite{lee1996quadratic} that in some cases the use of the Frobenius norm is preferred. We show this by means of an example borrowed from \cite{lee1996quadratic}. Suppose that $\Delta_Q=0$ and consider two possibilities for the value of $\Delta_P$:
\begin{equation}
\Delta_P^{(1)}=\begin{bmatrix}
1 & 0 \\
0 & 0
\end{bmatrix}\ \text{ and }\ \Delta_P^{(2)}=\begin{bmatrix}
1 & 0 \\
0 & 1
\end{bmatrix}.
\end{equation}
The matrix $\Delta_P^{(2)}$ has one more nonzero entry compared to $\Delta_P^{(1)}$, which can be seen as a less favorable uncertainty. However, the spectral norms of both $\Delta_P^{(1)}$ and $\Delta_P^{(2)}$ are equal, $\sigma_1(\Delta_P^{(1)})=\sigma_1(\Delta_P^{(2)})=1$. Thus, the difference between $\Delta_P^{(1)}$ and $\Delta_P^{(2)}$ cannot be inferred directly from their spectral norms. In this case, working with the Frobenius norm might be preferable as it reflects such differences in its value, i.e., $|\Delta_P^{(1)}|_\text{F}=1$ and $|\Delta_P^{(2)}|_\text{F}=2$. 

\subsection{Experiment design}

Provided that input-output data are given, one can obtain an estimation for the true system using \eqref{eq:chebcent_app}, and characterize the estimation error using \eqref{eq:Chebrad_app} and  \eqref{eq:inner_app}. However, in case there are no data (or the given data do not satisfy the desired worst-case error), one can think of collecting (more) data such that the desired accuracy is achieved. For this, given $\rho>0$ and a norm $\norm{\ \cdot\ }$, the goal is to find a finite-length input signal by which the data generated by the true system result in $\rad\mathcal{Z}_s(N) \leq \rho$. This is referred to as \emph{experiment design}. 

It was shown in \cite{coulson2022quantitative} that using some prior knowledge of the true system, for a given $\delta>0$, one can find a sequence of inputs such that the data generated by the true system satisfy
\begin{equation}
\sigma_{s}\left(\begin{bmatrix}
Y_-^\top & U_-^\top
\end{bmatrix}\right)\geq \delta.
\end{equation}
Here, we show that this result from \cite{coulson2022quantitative} can be used in experiment design for a desired identification accuracy with respect to arbitrary unitarily invariant norms. To that end, first, we study the relation between $\sigma_{s}\left(\begin{bmatrix}
Y_-^\top & U_-^\top
\end{bmatrix}\right)$ and $\rad\mathcal{Z}_s(N)$. 

\begin{proposition}
\label{prop:SNR_1}
Suppose that $\begin{bmatrix}
Y_-^\top & U_-^\top
\end{bmatrix}$ has full column rank. Then, for a unitarily invariant norm with the associated symmetric gauge \mbox{function $g$}, we have 
\begin{equation}
\label{eq:prop_rad_upper}
\rad\mathcal{Z}_s(N)\leq \frac{\sqrt{\sigma_1(\Phi|\Phi_{22})}}{\sigma_{s}\left(\begin{bmatrix}
Y_-^\top & U_-^\top
\end{bmatrix}\right)}g(\mathbf{1}),
\end{equation}
where $\mathbf{1}\in\mathbb{R}^p$ denotes the vector with all entries equal to $1$. 
\end{proposition}
\begin{proof}
We observe that
\begin{equation}
N|N_{22}=\Phi|\Phi_{22}+\begin{bmatrix}
I \\ Y_+^\top
\end{bmatrix}^\top \begin{bmatrix}
\Phi_{12} \\ \Phi_{22}
\end{bmatrix}\Psi\begin{bmatrix}
\Phi_{12} \\ \Phi_{22}
\end{bmatrix}^\top\begin{bmatrix}
I \\ Y_+^\top
\end{bmatrix},
\end{equation}
where
\begin{equation}
\Psi=\Phi_{22}^{-1}-\begin{bmatrix}
Y_- \\ U_-
\end{bmatrix}^\top\left(\begin{bmatrix}
Y_- \\ U_-
\end{bmatrix}\Phi_{22}\begin{bmatrix}
Y_- \\ U_-
\end{bmatrix}^\top\right)^{-1}\begin{bmatrix}
Y_- \\ U_-
\end{bmatrix}.
\end{equation}
Since $\Phi_{22}<0$, using a Schur complement argument, one can see that $\Psi\leq 0$. Thus,
\begin{equation}
\label{eq:pf_N|N_{22}}
N|N_{22}\leq\Phi|\Phi_{22}.
\end{equation}
This implies that $\sigma_1(N|N_{22})\leq\sigma_1(\Phi|\Phi_{22})$. Therefore, we have
\begin{equation}
\sigma(N|N_{22})\prec_w \sigma_1(N|N_{22}) \mathbf{1} \prec_w \sigma_1(\Phi|\Phi_{22}) \mathbf{1}.
\end{equation}
Hence,
\begin{equation}
\sigma((N|N_{22})^{\frac{1}{2}})\prec_w \sqrt{\sigma_1(\Phi|\Phi_{22})} \mathbf{1}.
\end{equation}
We further observe that the largest singular value of $(-N_{22})^{-\frac{1}{2}}$ is equal to $1/\sqrt{\sigma_s(-N_{22})}$. Thus, we have
\begin{equation}
\sigma_{[p]}((-N_{22})^{-\frac{1}{2}})\prec_w \frac{1}{\sqrt{\sigma_s(-N_{22})}} \mathbf{1}.
\end{equation}
Therefore,
\begin{equation}
\sigma_{[p]}((-N_{22})^{-\frac{1}{2}})\circ \sigma((N|N_{22})^{\frac{1}{2}})\prec_w \sqrt{\frac{\sigma_1(\Phi|\Phi_{22})}{\sigma_s(-N_{22})}} \mathbf{1}.
\end{equation}
Now, it follows from \eqref{eq:Chebrad_app} and Ky Fan's lemma \cite[Prop. 4.B.6]{marshall2011inequalities} that
\begin{equation}
\rad\mathcal{Z}_s(N)\leq \sqrt{\frac{\sigma_1(\Phi|\Phi_{22})}{\sigma_s(-N_{22})}}g(\mathbf{1}).
\end{equation}
This, together with $\sigma_{s}\left(\begin{bmatrix}
Y_-^\top & U_-^\top
\end{bmatrix}\right)=\sqrt{\sigma_s(-N_{22})}$, completes the proof.
\end{proof}

As an example, suppose that the norm is the spectral norm. In addition, suppose that the noise model is given by $\Phi_{11}=\varepsilon^2 I$, $\Phi_{12}=0$, and $\Phi_{22}=-I$. In that case, we have $g(\mathbf{1})=1$ and $\sigma_1(\Phi|\Phi_{22})=\varepsilon^2$. It follows now from \eqref{eq:prop_rad_upper} that
\begin{equation}
\rad\mathcal{Z}_s(N)\leq \frac{\varepsilon}{\sigma_{s}\left(\begin{bmatrix}
Y_-^\top & U_-^\top
\end{bmatrix}\right)}.
\end{equation}
For this particular case, a similar observation was also made in \cite[Ex. 1]{coulson2022quantitative}.

Inequality \eqref{eq:prop_rad_upper} represents the effect of a \emph{signal-to-noise ratio} on the identification accuracy with respect to an arbitrary unitarily invariant norm. It follows now from this result that, in the presence of noise, the desired identification accuracy $\rad\mathcal{Z}_s(N)\leq\rho$ is met if the data satisfy
\begin{equation}
\label{eq:rad_s2n-1}
 \sigma_{s}\left(\begin{bmatrix}
Y_-^\top & U_-^\top
\end{bmatrix}\right) \geq \frac{\sqrt{\sigma_1(\Phi|\Phi_{22})}}{\rho}g(\mathbf{1}).
\end{equation}
Having this relation, one can use \cite[Thm. 6]{coulson2022quantitative} to design an experiment for identification with a desired accuracy. This result quantifies how large $\sigma_{s}\left(\begin{bmatrix}
Y_-^\top & U_-^\top
\end{bmatrix}\right)$ should be in order to achieve a desired identification accuracy. 

The role of the smallest singular value of the data matrix $\begin{bmatrix}
Y_-^\top & U_-^\top
\end{bmatrix}$ in the identification accuracy is revealed in Proposition \ref{prop:SNR_1}. Now, the following proposition shows that the largest singular value of this matrix also contributes to the volume of the set of data-consistent systems by affecting the radius of the largest inner ball. 

\begin{proposition}
\label{prop:SNR_2}
Suppose that $\begin{bmatrix}
Y_-^\top & U_-^\top
\end{bmatrix}$ has full column rank and the norm is unitarily invariant and normalized. Then, we have 
\begin{equation}
\label{eq:prop_rad_in_upper}
\rad_\textup{in}\mathcal{Z}_s(N)\leq \frac{\sqrt{\sigma_1(\Phi|\Phi_{22})}}{\sigma_{1}\left(\begin{bmatrix}
Y_-^\top & U_-^\top
\end{bmatrix}\right)}.
\end{equation}
\end{proposition}
\begin{proof}
It follows from \eqref{eq:pf_N|N_{22}} that $\sigma_*(N|N_{22})\leq\sigma_1(N|N_{22})\leq \sigma_1(\Phi|\Phi_{22})$. We also observe that
\begin{equation}
\sigma_s((-N_{22})^{-\frac{1}{2}})=\frac{1}{\sigma_1((-N_{22})^{\frac{1}{2}})}=\frac{1}{\sigma_{1}\left(\begin{bmatrix}
Y_-^\top & U_-^\top
\end{bmatrix}\right)}.
\end{equation}
Substituting this into \eqref{eq:inner_app} yields \eqref{eq:prop_rad_in_upper}.
\end{proof}

\subsection{Examples}
\label{sec:VI}

In what follows, we consider two examples\footnote{The MATLAB codes for the examples of this section can be found at \href{https://github.com/a-shakouri/set-membership-system-identification}{https://github.com/a-shakouri/set-membership-system-identification}.}. The first one is a single-input single-output (SISO) system where the set of data-consistent systems is an ellipsoid in $\mathbb{R}^4$ allowing for a better demonstration of the introduced notions. The second example is a MIMO system, by which the effect of signal-to-noise ratio (SNR) on the Chebyshev radii with respect to several norms is shown through simulations. 

\subsubsection{Mass-spring-damper}

Consider the model of the mass-spring-damper system in Fig. \ref{fig:msd}, discretized using the zero-order hold method with a sample time of $1\ \textup{s}$. Here, $y$ denotes the position of the mass with respect to the equilibrium, and $u$ denotes the external force. The true parameters of the system are $m=2\ \textup{kg}$, $k=3\ \textup{N/m}$, and $c=1\ \textup{Ns/m}$. The mass position $y(t)$ is measured through a noisy experiment during a time horizon of $5$ seconds. For such a system, we consider the following model:
\begin{equation}
y(t+2)=P_0 y(t)+P_1 y(t+1)+Q_0 u(t) +Q_1 u(t+1)+w(t).
\end{equation}
The true system coefficients are $P_0=-0.6065$, $P_1=0.5659$, $Q_0=0.1583$, and \mbox{$Q_1=0.1886$}. The process noise is assumed to satisfy the energy bound \mbox{$\sum_{t=0}^3(w(t)-0.005)^2\leq 10^{-4}$} that can be captured by \eqref{eq:ass1-1}, $W_-^\top\in\mathcal{Z}_4(\Phi)$ and $\Phi\in\pmb{\Pi}_{1,4}^-$, with $\Phi_{11}=0$, $\Phi_{22}=-I_4$, and $\Phi_{12}=\Phi_{21}^\top=0.005\begin{bmatrix}
1 & 1 & 1 & 1
\end{bmatrix}$. 

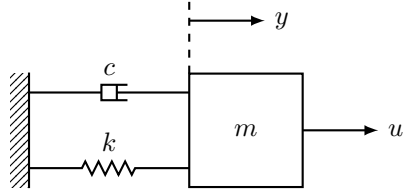
\begin{figure}[h]
\centering
\begin{tikzpicture}
 \tikzstyle{spring}=[thick,decorate,decoration={zigzag,pre length=0.7cm,post length=0.7cm,segment length=5}]
 \tikzstyle{damper}=[thick,decoration={markings, mark connection node=dmp, mark=at position 0.5 with 
   {
     \node (dmp) [thick,inner sep=0pt,transform shape,rotate=-90,minimum width=5pt,minimum height=5pt,draw=none] {};
     \draw [thick] ($(dmp.north east)+(4pt,0)$) -- (dmp.south east) -- (dmp.south west) -- ($(dmp.north west)+(4pt,0)$);
     \draw [thick] ($(dmp.north)+(0,-3pt)$) -- ($(dmp.north)+(0,3pt)$);
   }
 }, decorate]
 \tikzstyle{ground}=[fill,pattern=north east lines,draw=none,minimum
 width=1pt,minimum height=0.3cm]
 \node[draw,ground,thick] (M1) [minimum width=0.1cm, minimum height=1.5cm] {};
 \node[draw,outer sep=0pt,thick] (M2) at (3,0) [minimum width=1.5cm, minimum height=1.5cm] {$m$};
 \draw[spring] ($(M1.east) - (0,0.5)$) -- ($(M2.west) - (0,0.5)$) 
 node [midway,above=0.1cm] {$k$};
 \draw[damper] ($(M1.east) + (0,0.5)$) -- ($(M2.west) + (0,0.5)$)
 node [midway,above=0.1cm] {$c$};
 \draw[thick] ($(M1.north east)$) -- ($(M1.south east)$);
 \draw[thick, dashed] ($(M2.north west)$) -- ($(M2.north west) + (0,1)$);
 \draw[thick, -latex] ($(M2.north west) + (0,0.7)$) -- ($(M2.north west) + (1,0.7)$) node [right] {$y$};
  \draw[thick, -latex] ($(M2.east) + (0,0)$) -- ($(M2.east) + (1,0)$) node [right] {$u$};
 \end{tikzpicture}
	\caption{Mass-spring-damper system for Example 1.}
	\label{fig:msd}
\end{figure}

During the experiment, starting from $y(0)=0$, $y(1)=1$, we apply the input sequence $u(0)=1$, $u(2)=0$, $u(3)=-1$, $u(4)=0$, and $u(5)=1$, which makes $U_-$ equal to
\begin{equation}
U_-=\begin{bmatrix}
1 & 0 & -1 & 0 \\
 0 & -1 & 0 & 1 \\
\end{bmatrix}.
\end{equation}
Now, consider a simulation with the noise sequence
\begin{equation}
W_-=\begin{bmatrix}
0.0105 & -0.0013 & 0.0092 & 0.0084
\end{bmatrix},
\end{equation}
for which the collected output data are
\begin{equation}
\begin{split}
Y_-&=\begin{bmatrix}
0 & 1 & 0.7347 & -0.3807 \\
1 & 0.7347 & -0.3807 & -0.8101
\end{bmatrix}, \\
Y_+&=\begin{bmatrix}
0.7347 & -0.3807 & -0.8101 & -0.0305
\end{bmatrix}.
\end{split}
\end{equation}
For this case, the Chebyshev center \eqref{eq:chebcent_app} is
\begin{equation}
\hat{P}=\begin{bmatrix}
-0.6065 & 0.5659
\end{bmatrix},\hspace{0.25cm} \hat{Q}=\begin{bmatrix}
0.1583 & 0.1886
\end{bmatrix},
\end{equation}
and the Chebyshev radius in terms of the Euclidean norm is $\rad\mathcal{Z}_4(N)=0.0458$. For the described simulation, the set of data-consistent systems, the true system, and the Chebyshev center \eqref{eq:chebcent_app} are shown in Fig. \ref{fig:ex1}. 

To visualize the results beyond a single simulation, we obtain 1000 simulations for uniformly distributed random sequences of noise normalized to comply with \mbox{$W_-^\top\in\mathcal{Z}_4(\Phi)$}. The Chebyshev centers \eqref{eq:chebcent_app} are obtained for each simulation and are shown by the gray points in Fig. \ref{fig:ex2}.

\begin{figure}
    \centering
    \includegraphics[width=0.75\columnwidth]{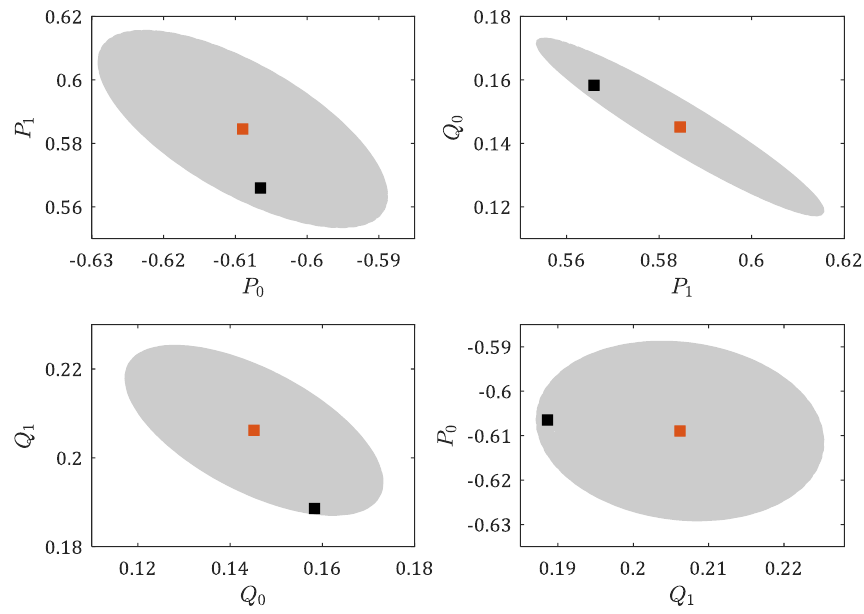}
    \caption{The set of data-consistent systems is shown by the gray area. The red and black squares show the Chebyshev center and the true system, respectively.}
    \label{fig:ex1}
\end{figure}

\begin{figure}
    \centering
    \includegraphics[width=0.75\columnwidth]{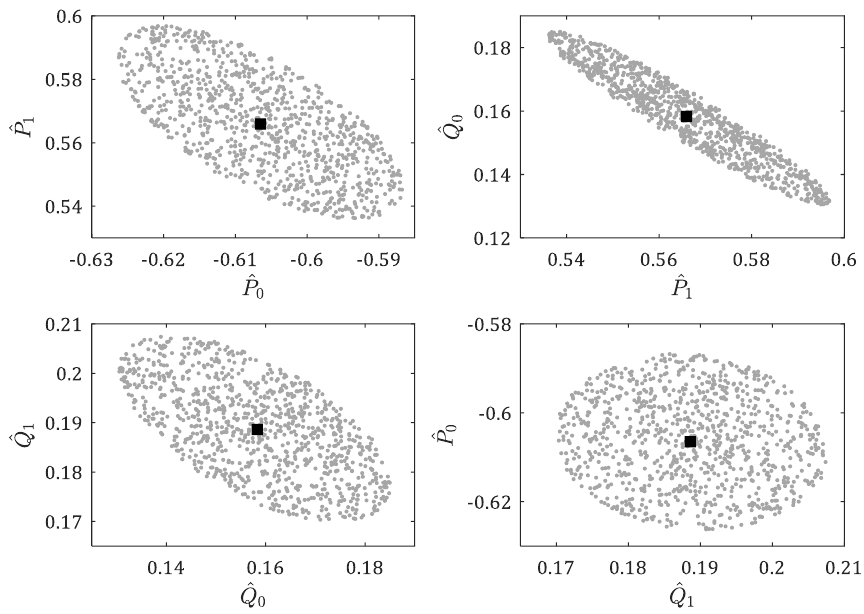}
    \caption{The gray points show the Chebyshev centers obtained for different sequences of the noise. The black square shows the true system.}
    \label{fig:ex2}
\end{figure}

\subsubsection{Resistor-capacitor circuit}

Consider the model of the resistor-capacitor (RC) circuit in Fig. \ref{fig:ex3}, discretized using the zero-order hold method with a sample time of $1\ \textup{s}$. Here, the output variables $y_1$, $y_2$, and $y_3$ are the voltages across the capacitors $C_1$, $C_2$, and $C_3$, respectively, and the input variables $u_1$ and $u_2$ are the voltage sources. The true parameters of the system are $R_1=7\ \Omega$, $R_2=5\ \Omega$, $R_3=10\ \Omega$, $R_4=15\ \Omega$, $C_1=0.5\ \textup{F}$, $C_2=0.4\ \textup{F}$, and $C_3=0.6\ \textup{F}$. We collect the input-output data for a duration of $10$ seconds, starting from $y(0)=0$, where the noise sequence is assumed to satisfy $W_-^\top\in\mathcal{Z}_{10}(\Phi)$ and $\Phi\in\pmb{\Pi}_{3,10}^-$ with $\Phi_{11}=10^{-6} I_3$, $\Phi_{22}=-I_{10}$, and $\Phi_{12}=\Phi_{21}^\top=0$. Let the SNR between the input and the noise be defined as $\sigma_m(U_-)/\sigma_1(W_-)$. We run $1000$ simulations per SNR for a range of SNR from $1$ to $10^3$. For each simulation, uniformly distributed random sequences of noise and input are generated and then normalized, respectively, to attain the required noise bound $\sigma_1(W_-)\leq10^{-3}$ and the SNR value. 

Fig. \ref{fig:ex3-2} shows the Chebyshev radii with respect to the nuclear, spectral, and Frobenius norms. Fig. \ref{fig:ex3-3} shows the radius of the largest ball within the set of data-consistent systems, $\rad_\textup{in} \mathcal{Z}_s(N)$, which has the same value for nuclear, spectral, and Frobenius norms as they are all normalized, i.e., they satisfy $g(e_1)=1$. In both figures, the solid curves are the mean values, and the shaded areas indicate the standard deviation computed through 1000 simulations per SNR. These simulations demonstrate the results of Propositions~\ref{prop:SNR_1}~and~\ref{prop:SNR_2}, where it was shown that both $\rad \mathcal{Z}_s(N)$ and $\rad_\text{in} \mathcal{Z}_s(N)$ admit upper bounds in terms of an SNR measure. 

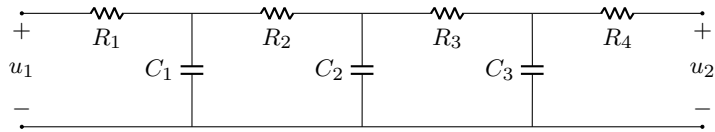
\begin{figure}[h]
	\centering
	\begin{circuitikz}[american, scale = 1.5, font=\scriptsize]
		\ctikzset{bipoles/length=5mm} 
		\draw (0,1) to[open,v=$u_{1}$] (0,0);
		\draw (0,1) to[short, *-] (0,1);
        \draw (0,1) to [R, V_=$R_1$,R] (1.5,1);
        \draw (1.5,1) to [R, V_=$R_2$,R] (3,1);
        \draw (3,1) to [R, V_=$R_3$,R] (4.5,1);
        \draw (4.5,1) to [R, V_=$R_4$, R] (6,1);
        \draw (6,1) to[short, -*] (6,1);
        \draw (6,1) to[open,v=$u_{2}$] (6,0);
        \draw (1.5,1) to [C, V_=$C_1$, C] (1.5,0);
        \draw (3,1) to [C, V_=$C_2$, C] (3,0);
        \draw (4.5,1) to [C, V_=$C_3$, C] (4.5,0);
        \draw (0,0) to[short, *-] (0,0);
        \draw (0,0) to (6,0);
        \draw (6,0) to[short, -*] (6,0);
	\end{circuitikz}
	\caption{RC circuit for Example 2.}
	\label{fig:ex3}
\end{figure} 

\begin{figure}
    \centering
    \includegraphics[width=0.75\columnwidth]{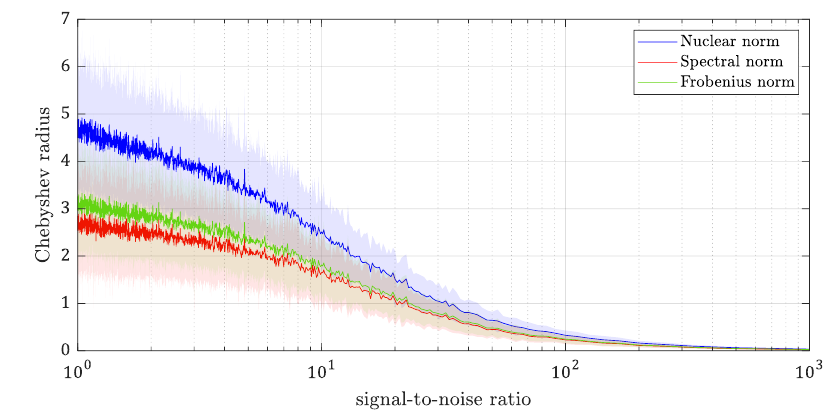}
    \caption{The Chebyshev radii with respect to different norms. The curves indicate the mean values with the shaded area indicating the standard deviation over 1000 simulations per SNR.}
    \label{fig:ex3-2}
\end{figure}

\begin{figure}
    \centering
    \includegraphics[width=0.75\columnwidth]{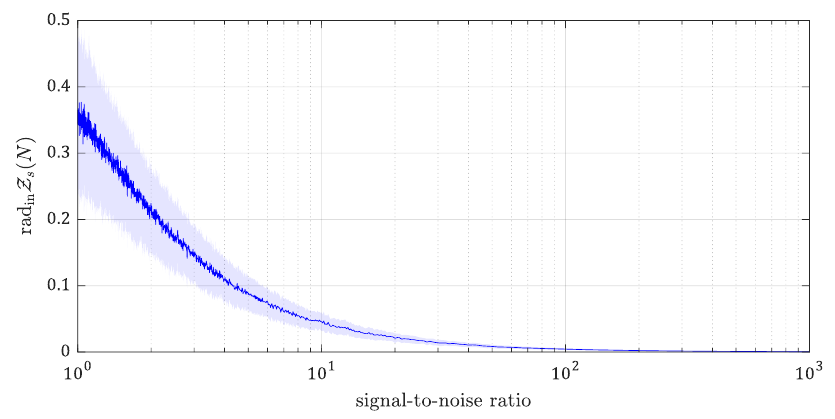}
    \caption{The radius of the largest ball within the set of data-consistent systems. The curves indicate the mean values with the shaded area indicating the standard deviation over 1000 simulations per SNR.}
    \label{fig:ex3-3}
\end{figure}



\section{Conclusions}
\label{sec:VII}
In this work, we have formulated analytic expressions for Chebyshev centers, the Chebyshev radii, and the diameters of certain types of QMI-induced sets with respect to arbitrary unitarily invariant matrix norms. It has been shown that the presented Chebyshev center is the only common center among all matrix norms. We also presented the radius of a largest ball within a QMI-induced set. We discussed the applications of the provided results in data-driven modeling and control of unknown linear time-invariant systems. A notion that has not been fully studied in this work is the \emph{volume} of the QMI-induced sets, which can be a subject for future studies. Moreover, future work could focus on applications of the obtained results in experiment design for system identification with prescribed accuracy. 

\backmatter

\bmhead{Acknowledgements}

Henk J. van Waarde acknowledges financial support by the Dutch Research Council under the NWO Talent Programme Veni Agreement (VI.Veni.22.335).

\section*{Declarations}

Not applicable. 






\bibliography{sn-bibliography}

\end{document}